\documentclass[12pt,twoside]{article}
\usepackage{amsfonts}
\usepackage{cite}
 \usepackage{color}
\usepackage{amsmath} \numberwithin{equation}{section}
\usepackage{amsthm}
\topmargin by-\headsep  \textheight 24cm
\newtheorem{lemma}{Lemma}[section]

\newtheorem{definition}{Definition}[section]
\newtheorem{theorem}[definition]{Theorem}

\newtheorem{remark}{Remark}[section]

\begin{document}
\title{Vanishing Shear Viscosity Limit and Boundary Layer Study on the  Planar MHD system}
\author{Xulong Qin$^a$,\quad Tong Yang$^b$\thanks{E-mail address:   matyang@cityu.edu.hk},\quad Zheng-an Yao$^a$, \quad Wenshu Zhou$^c$ \\
  \small  a. Department of Mathematics, Sun Yat-sen University, Guangzhou 510275,  China\\
  \small  b. Department of Mathematics, Jinan University, Guangzhou, P. R. China;\\
  \small \& Department of Mathematics, City University of Hong Kong, Hong Kong, China\\
  \small c. Department of Mathematics, Dalian Minzu University,  Dalian 116600, China\\}


\date{}
 \maketitle

\begin{abstract}
\small{We consider an initial boundary problem for the planar MHD system under the general condition on the heat conductivity $\kappa$
that may depend on both the density $\rho$ and the temperature $\theta$  satisfying
$
\kappa(\rho,\theta)\geq\kappa_1 \theta^{q}$ for some
constants $\kappa_1>0$ and $q>0.$
Firstly, the global existence of strong solution for large initial
data is obtained,  and   then the limit of the vanishing shear viscosity
is justified. In addition,   the   $L^2$ convergence rate  is obtained together with the estimation on the thickness of the boundary layer. }

\medskip
 {\bf Keywords}.  MHD system, global existence, vanishing shear viscosity, boundary layer.

\medskip
 {\bf 2010 MSC}. 35B40; 35B45; 76N10; 76N20;  76W05; 76X05.
\end{abstract}

%
\bigbreak

\section{Introduction}
 The planar Magnetohydrodynamics (MHD) system  with constant longitudinal magnetic field  is governed by the following equations:

 \begin{equation}\label{e1}
  \left\{\begin{split}
 &\rho_t+(\rho u)_x=0,\\
 &(\rho u)_t+\left(\rho u^2+p+\frac12 |\mathbf{b}|^2\right)_x
 =(\lambda u_{x})_x,\\[1mm]
 &(\rho \mathbf{w})_t+(\rho u \mathbf{w}-\mathbf{b})_x=(\mu\mathbf{w}_{x})_x,\\[1mm]
 &\mathbf{b}_t+(u\mathbf{b}-\mathbf{w})_x=(\nu\mathbf{b}_{x})_x,\\[1mm]
& (\rho e)_t+(\rho  u e)_x-(\kappa e_x)_x+pu_x= \lambda u_x^2+\mu
|\mathbf{w}_x|^2+\nu |\mathbf{b}_x|^2.
 \end{split}\right.
 \end{equation}
 Here  $\rho$ denotes the density, $\theta$  the temperature,  $u\in \mathbb{R}$
  the longitudinal velocity, $\mathbf{w}=(w_1,w_2)\in \mathbb{R}^2$ the transverse velocity,
  $\mathbf{b}=(b_1,b_2)\in \mathbb{R}^2$ the transverse magnetic field, $p=p(\rho,\theta)$ the pressure,
  $e=e(\rho,\theta)$ the internal energy, and $\kappa=\kappa(\rho,\theta)$   the heat conductivity respectively.
 The coefficients $\lambda, \mu$ and $\nu$ are assumed to be positive constants, where $\lambda$ and $\mu$ are the viscosity coefficients, and $\nu$ is the magnetic diffusivity. And the state equations are
 \begin{equation}\label{e2}
 p=\gamma\rho \theta, \qquad e=c_v\theta,
 \end{equation}
with constants $\gamma>0$ and $c_v>0$. Without loss of
generality, set $c_v=1$.  Based on some
physical models in which  $\kappa$ grows like $\theta^q $, for example,
 $q \in [4.5, 5.5]$ for molecular diffusion in gas (see \cite{ZYY}), we assume that $\kappa=\kappa(\rho,\theta)$ is twice
differential in $\mathbb{R}^+\times\mathbb{R}^+$ and satisfies
\begin{equation}\label{kappa}
\kappa(\rho,\theta)\geq \kappa_1 \theta^q \quad\hbox{\rm with
 constants}~ \kappa_1>0 ~\hbox{\rm and}~ q>0.
\end{equation}

In this paper, we consider system \eqref{e1} in a bounded domain $Q_T=\Omega\times (0, T)$ with $\Omega=(0, 1)$
under the following initial and boundary conditions:
\begin{equation}\label{e4}
\left\{
\begin{array}{lllllll}
(\rho,u,\theta, \mathbf{w},\mathbf{b})(x,0)=(\rho_0,u_0,\theta_0,\mathbf{w}_0,\mathbf{b}_0)(x),\\[1mm]
(u, \mathbf{b},\theta_x)|_{x=0, 1}=\mathbf{0},\quad \mathbf{w}(0,t)=\mathbf{w}^-(t),\quad\mathbf{w}(1,t)=\mathbf{w}^+(t).
\end{array}
\right.
\end{equation}
We aim to study the global existence, vanishing shear viscosity
limit, convergence rate and boundary layer effect
 of solutions to problem \eqref{e1}-\eqref{e4} with large initial data under the  condition \eqref{kappa}.

Because of its physical importance and
mathematical challenge, the MHD system  has been extensively
studied, see
\cite{B,CW1,CW2,Wang,CD,HW2,VH,KO,LL,LZ} and the references
therein.  Without  magnetic effect, MHD system is
reduced to the compressible Navier-Stokes equations that are better understood
mathematically.
For example, in  one space dimension,
  there is a seminal work  by
Kazhikhov and Shelukhin \cite{KS} on the global existence of strong
solutions for the compressible Navier-Stokes equations with constant coefficients  and large initial data. However, the corresponding result for the
MHD system with constant   coefficients remains well known unsolved problem.

 On the other hand, as for
  the well-posedness theory of MHD,  Vol'pert and Hudjaev \cite{VH}  firstly proved the existence and uniqueness of local smooth solutions, and then
 the global existence of smooth solution with small initial data was established in \cite{KO}. In addition,
  under the following condition on $\kappa$:
 \begin{equation}\label{assumption4}
\begin{split}
 C^{-1}(1+\theta^q)\leq \kappa(\rho,\theta) \leq C (1+\theta^q),~~q>0,
\end{split}
\end{equation}
 the existence  of global solution to the problem \eqref{e1}-\eqref{e4} with large  initial
data was studied in \cite{CW1, Wang} for $q\geq 2$,\cite{FanJiang} for $q\geq 1$, and \cite{FHL,HJ} for $q>0$. In fact, the
condition like \eqref{assumption4}  was also used  in other
 papers, cf.  \cite{Wang,CW2,FanJiang,FHL,K,JZ} and references  therein.   In this paper, we will firstly
  show the global existence of strong solution to the problem \eqref{e1}-\eqref{e4}
under  the condition \eqref{kappa}.

 The problem of vanishing viscosity has been an interesting and
 challenging problem in many setting, in particular with boundary, for example, in the boundary layer theory (cf. \cite{Sch}).
And there are many mathematical results on this problem, cf.  \cite{S,JZ, FS,FS1,QYYZ,YaoZhangZhu}  for the work on
 Navier-Stokes equations, and \cite{FanJiang, FHL}  for the problem \eqref{e1}-\eqref{e4}. As a second result of this paper, we
 will justify such limit in term of vanishing of shear viscosity and
describe the  convergence of $\mathbf{w}$ and $\mathbf{b}$ under the condition \eqref{kappa}.

Now we briefly review some related works on the
  boundary layer theory that is  one of the fundamental
problems in fluid dynamics established by
Prandtl in 1904. Without the magnetic effect, Frid and Shelukhin  \cite{FS1}  investigated the
boundary layer effect of the compressible isentropic Navier-Stokes
equations  with cylindrical symmetry, and proved the existence of
boundary layers    thickness  in the order of  $O(\mu^{\alpha}) (0<\alpha<1/2)$.
Recently, this result was investigated
in a more general setting,  cf. \cite{JZ, QYYZ} for the non-isentropic case and
\cite{YaoZhangZhu} for the case with density-dependent viscosity. With the magnetic field, the authors in \cite{YZ} studied the problem on boundary layer for the isentropic planar MHD system with the constant initial data and obtained the same thickness of boundary layer estimate as   in \cite{FS1,JZ,QYYZ}.
 In this paper, we extend this result to problem \eqref{e1}-\eqref{e4}
 in general setting by introducing some new analytic technique in obtaining the
 $L^p$  estimate 
 on the second derivative of the velocity field by using
 theories for linear parabolic equations.


 In the following, some notation will be used. Firstly, denote $Q_t=\Omega\times(0, t)$ for $t\in (0,T]$. For  integer $k\geq0$,   constant $p\geq1$ and    $\mathcal{O}\subset\mathbb{R}^n$, $W^{k,p}(\mathcal{O})$ and $W_0^{k,p}(\mathcal{O})$  denote the usual Sobolev spaces.  $L^p(I,B)$ is the
space of all strong measurable, $p^{th}$-power
integrable (essentially bounded if $p=\infty$) functions from $I$ to
$B$, where $I\subset \mathbb{R}$ and $B$ is a Banach space. For
simplicity, we also use the notation $\|(f, g,
\cdots)\|^2_{B}=\|f\|^2_{B}+\|g\|_{B}^2+\cdots$ for  $f,
g,\cdots$ belonging to
  $B$ equipped with a norm $\|\cdot\|_{B}$.

The initial and boundary functions are assumed to satisfy
 \begin{equation}\label{assumption1}
\left\{\begin{split}
&\rho_0>0,\,\,\theta_0>0,\,\,\, \|(\rho_0^{-1},\theta_0^{-1})\|_{C(\overline{\Omega})}<\infty,
\,\,\|(\mathbf{w}^-,\mathbf{w}^+)\|_{C^1[0,T]}<\infty, \\[1mm]
&(\rho_0,\mathbf{w}_0,\theta_0)\in W^{1,2}(\Omega),\,\,\mathbf{b}_0\in W_0^{1,2}(\Omega),
\,\, u_0\in W_0^{1,2}(\Omega)\cap W^{2,m}(\Omega),~ m\in (1, +\infty),\\[1mm]
&u_0(1)=u_0(0)=0,\,\,\mathbf{w}_0(0)=\mathbf{w}^-(0),\,\,\mathbf{w}_0(1)=\mathbf{w}^+(0).
\end{split}\right.
\end{equation}

Then the first result of this paper  can be stated as
follows.

\begin{theorem}\label{existencethm}
Let \eqref{kappa} and \eqref{assumption1} hold. Then

(i)~~ For any fixed $\mu>0$, problem
\eqref{e1}--\eqref{e4} admits a unique strong solution
$(\rho,u,\mathbf{w},\mathbf{b},\theta)$  satisfying
\begin{equation*}
\begin{split}&\inf_{Q_T}\rho>0,~~\inf_{Q_T}\theta>0,\quad (\rho, u, \mathbf{w},\mathbf{b},\theta)\in L^{\infty}(Q_T),\quad\rho \in L^{\infty}(0,T;W^{1,2}(\Omega)),\quad \rho_t \in L^{2}(Q_T),\\
&(u, \mathbf{w},\mathbf{b},\theta) \in L^{\infty}(0,T;W^{1,2}(\Omega))\cap L^{2}(0,T;W^{2,2}(\Omega)),\quad (u_t, \mathbf{w}_t,\mathbf{b}_t,\theta_t) \in L^{2}(Q_T).
\end{split}
\end{equation*}
 Moreover, there exists a positive constant $C$ independent of $\mu$ such that
\begin{equation}\label{ve}
\begin{split}
&C^{-1}\leq \rho, \theta \leq C,\quad  \|(u,\mathbf{w},\mathbf{b})\|_{L^\infty(Q_T)}\leq C,\\[2mm]
&\|(\rho_t,\rho_x, u_x,\mathbf{b}_x,\theta_x)\|_{L^\infty(0, T;L^2(\Omega))}
+\|(u_t, \mathbf{b}_t,\theta_t, u_{xx}, \theta_{xx})\|_{L^2(Q_T)}\leq C,\\[2mm]
&\|\mathbf{w}_x\|_{L^\infty(0,
T;L^1(\Omega))}+\|\mathbf{w}_t\|_{L^2(Q_T)}\leq C,\\
&\mu^{1/4}\|\mathbf{w}_x\|_{L^{\infty}(0,T;L^2(\Omega))} +\mu^{3/4}\|\mathbf{w}_{xx}\|_{L^2(Q_T)} \leq C,\\
&\|\sqrt{\omega}\mathbf{w}_x\|_{L^\infty(0,
T;L^2(\Omega))}+\|\sqrt{\omega}\mathbf{b}_{xx}\|_{L^2(Q_T)} \leq C,\\
\end{split}
\end{equation}
where $\omega: [0, 1]\rightarrow [0, 1]$ is defined by
\begin{equation*}
\begin{aligned}
\omega(x)=\left\{\begin{aligned} &x,&0\leq x\leq  1/2,\\ &1-x,&
1/2\leq x\leq 1.
\end{aligned}\right.
\end{aligned}
\end{equation*}

(ii)~~There exist  functions $(\overline\rho, \overline u,
\overline{\mathbf{w}}, \overline{\mathbf{b}},\overline\theta)$ in the family $\mathbb{F}$ defined by
\begin{equation*}\label{base0}
\mathbb{F}:\left\{\begin{split}
&  \overline\rho, \overline\theta>0,\quad(\overline u, \overline{\mathbf{b}})|_{x=0, 1}=0,\\
&(\overline\rho,  1/\overline\rho,\overline u, \overline{\mathbf{w}}, \overline{\mathbf{b}},
\overline\theta,1/\overline\theta)\in L^\infty(Q_T),\quad
  \overline{\mathbf{w}} \in L^\infty(0, T;W^{1,1}(\Omega)),\\
    &(\overline\rho_t,\overline\rho_x,\overline u_x,\overline{\mathbf{b}}_x,\overline\theta_x)
    \in L^\infty(0, T;L^2(\Omega)),\quad (\overline u_x,\overline{\mathbf{b}}_x,\overline \theta_x) \in L^2(0, T;L^\infty(\Omega)),\\
   &\big(\overline u_t,  \overline{\mathbf{w}}_t, \overline{\mathbf{b}}_t,\overline\theta_t,
   \overline u_{xx}, \overline\theta_{xx}\big) \in L^2(Q_T),\\
       &\sqrt{\omega}\overline{\mathbf{w}}_x\in L^\infty(0, T;L^2(\Omega)),\quad \sqrt{\omega}\overline{\mathbf{b}}_{xx}  \in  L^2(Q_T),\\
 \end{split}\right.
\end{equation*}
such that as $\mu\rightarrow 0$
\begin{equation*}\label{rate}
 \begin{split}
 &(\rho,u,\mathbf{b}, \theta)\rightarrow (\overline\rho,\overline u,\overline{\mathbf{b}},
 \overline\theta)~~ \hbox{\rm  in}~~C^\alpha(\overline Q_T),~\forall \alpha\in(0, 1/4),\\
 &(u_x, \mathbf{b}_x, \theta_x)\rightarrow (\overline u_x, \overline{\mathbf{b}}_x,\overline\theta_x)~~\hbox{\rm strongly in}
 ~~L^{2}(Q_T),\\
    &  (\rho_t,\rho_x) \rightharpoonup (\overline\rho_t, \overline\rho_x)
    ~~\hbox{\rm weakly}-*~\hbox{\rm in}~ L^\infty(0, T; L^2(\Omega)),\\
  &(u_t,\mathbf{b}_t,\theta_t,u_{xx}, \theta_{xx})
  \rightharpoonup(\overline u_t,\overline{\mathbf{b}}_t, \overline\theta_t,\overline u_{xx},
  \overline\theta_{xx})~~ \hbox{\rm weakly in}~~L^2(Q_T),\\
  & \mathbf{b}_{xx}  \rightharpoonup  \overline{\mathbf{b}}_{xx}
  \quad \hbox{\rm  weakly in}~~L^2\big((\delta,1-\delta)\times(0, T)\big),~\forall \delta\in\big(0, 1/2\big),\\
      \end{split}
  \end{equation*}
 and
  \begin{equation*}
   \begin{split}
 & \mathbf{w}  \rightarrow   \overline{\mathbf{w}}~~~~\hbox{\rm  in}~~C^\alpha([\delta,1-\delta]\times[0, T]),~\forall \delta\in\big(0, 1/2\big),~\alpha\in(0, 1/4),\\
 &\mathbf{w}_t   \rightharpoonup   \overline{\mathbf{w}}_t~~\hbox{\rm weakly in}~~L^2(Q_T),\\
  &\mathbf{w}_x \rightharpoonup  \overline{\mathbf{w}}_x~\hbox{\rm weakly}-*~\hbox{\rm in}~ L^\infty(0, T; L^2(\delta,1-\delta)),~\forall \delta\in\big(0, 1/2\big),\\
  &\mathbf{w}\rightarrow  \overline{\mathbf{w}}~~ \hbox{\rm strongly in}~~L^r(Q_T),\quad\forall r \in [1, +\infty),\\
  &\sqrt{\mu}\|\mathbf{w}_x\|_{L^2(Q_T)} \rightarrow 0.
   \end{split}
    \end{equation*}
 Moreover, $(\overline{\rho},\overline{u},\overline{\mathbf{w}},\overline{\mathbf{b}},\overline{\theta})$ is the unique solution of problem  \eqref{e1}--\eqref{e4} with $\mu=0$ in $\mathbb{F}$.

(iii)~~Let
$(\overline{\rho},\overline{u},\overline{\mathbf{w}},\overline{\mathbf{b}},\overline{\theta})
\in \mathbb{F}$ be a solution for problem \eqref{e1}--\eqref{e4} with $\mu=0$. Then
\begin{equation*}\label{0u3}
\begin{aligned}
&\|(\rho-\overline\rho, u-\overline u, \mathbf{w}-\overline{\mathbf{w}},\mathbf{b}-\overline{\mathbf{b}},
\theta-\overline\theta )\|_{L^\infty(0, T;L^2(\Omega))}\\
&\quad\quad\quad+\|(u_x-\overline u_x,  \mathbf{b}_x-\overline{\mathbf{b}}_x, \theta_x-\overline\theta_x)\|_{L^2(Q_T)}=
 O(\mu^{1/4}).
\end{aligned}
\end{equation*}

\end{theorem}
\begin{remark}
Following the argument in \cite{KS}(cf. \cite{CW1}),   \eqref{ve}
implies   that problem \eqref{e1}--\eqref{e4} admits a unique classical solution if the initial data is sufficiently smooth.
\end{remark}


We now present a sketch of the proof of \eqref{ve}. Firstly, the
uniform upper and lower bounds of the density can be obtained as
in the previous literatures, cf. \cite{FanJiang,FHL,HJ,YZ}.
A key observation in the paper  is to obtain
  a uniform bound on
$\|u_{xx}\|_{L^{m_0}(Q_T)} (m_0>1)$ that
can be obtained by the $L^p$-theory of
linear parabolic equations (see Lemma \ref{2.5}).
In fact, by using this estimate and some delicate analysis, we then
deduce the  key estimates of the bounds on $\|\omega\mathbf{w}_x\|_{L^\infty(0,
T;L^2(\Omega))}$ and
 $\|(u_t,\mathbf{b}_t,\mathbf{w}_t,
u_{xx},\theta_{x},\omega\mathbf{b}_{xx})\|_{L^{2}(Q_T)}$(see Lemma \ref{2.10}).
In this step, the difficulty caused by the coupling between the transverse velocity and transverse magnetic field is overcome.
 With these uniform estimates with respect  to $\mu$, 
 the uniform bounds on $\|\sqrt{\omega}\mathbf{w}_x\|_{L^\infty(0, T;L^2(\Omega))}$
 and $\big(\mu^{1/4}\|\mathbf{w}_x\|_{L^\infty(0,T;L^2(\Omega))}+\mu^{3/4}\|\mathbf{w}_{xx}\|_{L^2(Q_T)}\big)$
 (see Lemma \ref{2.13}) can then be obtained that are
 essential to the estimation on
  both  convergence rate and boundary layer thickness. In addition,  an upper bound on
$\theta$ follows (see Lemma \ref{2.14}) together with the uniform bound on $\|(\theta_t,
\theta_{xx})\|_{L^{2}(Q_T)}$ (see Lemma \ref{2.15}).

The next result of this paper is about
the estimation on  the thickness of  boundary layer. For this, we
first recall
the definition of a BL-thickness, cf.
\cite{FS1}, as follows
\begin{definition}\label{defination}
A function $\delta(\mu)$ is called a BL-thickness for problem
\eqref{e1}-\eqref{e4} with vanishing  $\mu$ if
$\delta(\mu)\downarrow 0$ as $ \mu \downarrow 0$, and
\begin{equation*}\label{12}
\begin{aligned}
&\lim\limits_{\mu\rightarrow 0}\|(\rho-\overline\rho, u-\overline u,
\mathbf{w}-\overline{\mathbf{w}}, \mathbf{b}-\overline{\mathbf{b}},
\theta-\overline\theta)\|_{L^\infty(0,T;L^\infty(\delta(\mu),1-\delta(\mu))}=0,\\
&\mathop{ \inf\lim}\limits_{\mu\rightarrow 0}\|(\rho-\overline\rho,
u-\overline u, \mathbf{w}-\overline{\mathbf{w}},
\mathbf{b}-\overline{\mathbf{b}},\theta-\overline\theta)\|_{L^\infty(0,T;L^\infty(\Omega))}>0,
\end{aligned}
\end{equation*}
where   $(\rho, u, \mathbf{w}, \mathbf{b}, \theta)$ and
$(\overline\rho,\overline u, \overline{\mathbf{w}},
\overline{\mathbf{b}}, \overline\theta)$ are the solutions to
problem \eqref{e1}-\eqref{e4}
with $\mu>0$ and $\mu=0$, respectively.
\end{definition}

The second result of this paper is

\begin{theorem}\label{am} Let  the assumptions in Theorem \ref{existencethm}  hold.  Then  any function $\delta(\mu)$ satisfying  $\delta(\mu)\downarrow 0$ and $\frac{\sqrt{\mu}}{\delta(\mu)}\rightarrow 0$ as $ \mu \downarrow 0$
is  a BL-thickness for problem \eqref{e1}-\eqref{e4} such that
\begin{equation*}
\begin{split}
   &\lim\limits_{\mu\rightarrow 0} \|(\rho-\overline\rho,u-\overline u,\mathbf{b}-\overline{\mathbf{b}},\theta-\overline\theta)\|_{C^\alpha(\overline Q_T)}=0,\quad \forall \alpha\in (0,1/4),\\
   &\lim\limits_{\mu\rightarrow 0} \|\mathbf{w}-\overline{\mathbf{w}}\|_{L^\infty(0, T;L^\infty(\delta(\mu),
   1-\delta(\mu)))}=0,\quad \mathop{\inf\lim}\limits_{\mu\rightarrow 0} \|\mathbf{w}-\overline{\mathbf{w}}\|_{L^\infty(0, T;L^\infty(\Omega))}>0,
  \end{split}
\end{equation*}
when  $(\mathbf{w}^-(t), \mathbf{w}^+(t)) \not\equiv(\overline{\mathbf{w}}(0,t), \overline{\mathbf{w}}(1,t))$.
\end{theorem}

The rest of this paper is organized as follows. In Section
2, we will prove Theorem \ref{existencethm}.   The proof  of Theorem  1.3 will be given in Section 3.

\section{Proof of Theorem 1.1}

The existence and uniqueness of local solution can be obtained by
using the Banach theorem and the contractivity of the operator
through the linearization of the system, cf.\cite{Wang,Nash}. Then
to obtain global solution,  we only need to close the
 a priori estimates of solutions. The next subsection  is about
  deriving the $\mu$-uniform estimates given in \eqref{ve}.
From now on,    we use $C$ to denote a positive generic constant independent of $\mu$.

\subsection{A priori estimates  independent of $\mu$}

Firstly, rewrite  \eqref{e1} as
\begin{equation}\label{e20}
\begin{split}
&\mathcal{E}_t+\Big[u\big(\mathcal{E}+p+\frac12|\mathbf{b}|^2\big)-\mathbf{w}\cdot\mathbf{b}\Big]_x=\big(\lambda uu_x+\mu\mathbf{w}\cdot\mathbf{w}_x+\nu\mathbf{b}\cdot\mathbf{b}_x+\kappa\theta_x\big)_x,\\
&(\rho \mathcal{S})_t+(\rho u \mathcal{S})_x-\left(\frac{\kappa\theta_x}{\theta}\right)_x=\frac{\lambda u_x^2+\mu|\mathbf{w}_x|^2+\nu|\mathbf{b}_x|^2}{\theta}
+\frac{\kappa\theta_x^2}{\theta^2},
\end{split}
\end{equation}
where $\mathcal{E}$ and $\mathcal{S}$ are the total energy and the
entropy, respectively, given by
\begin{equation*}
\begin{split}
&\mathcal{E}=\rho\Big[\theta+\frac12(u^2+|\mathbf{w}|^2)\Big]+\frac12|\mathbf{b}|^2,\quad\mathcal{S}
=\ln\theta-\gamma\ln\rho.
\end{split}
\end{equation*}

\begin{lemma}\label{2.1}
\label{energy}Under the assumptions in Theorem \ref{existencethm}, we have
\begin{equation}\label{ba1}
\begin{split}
&\int_\Omega \rho(x,t)dx=\int_\Omega \rho_0(x)dx,\quad\forall t \in (0, T),\\
&\sup\limits_{0<t<T}\int_\Omega\big[\rho(\theta+ u^2+|\mathbf{w}|^2)+ |\mathbf{b}|^2\big] dx\leq
C,\\[1mm]
& \iint_{Q_T}
\left(\frac{\lambda u_x^2+\mu|\mathbf{w}_x|^2+\nu|\mathbf{b}_x|^2}{\theta}+\frac{\kappa\theta_x^2}{\theta^2}\right)dxdt
\leq C.
\end{split}
\end{equation}
\end{lemma}
\begin{proof}
Integrating $\eqref{e20}_1$ over $Q_t=\Omega\times(0, t)$  yields
\begin{equation}\label{total}
\begin{split}
&\int_\Omega \mathcal{E}dx=\int_\Omega \mathcal{E}|_{t=0}dx+\mu\int_0^t(\mathbf{w}\cdot\mathbf{w}_x)|_{x=0}^{x=1}ds.
\end{split}
\end{equation}
To estimate the final integral on the right hand side of \eqref{total}, we first integrate  \eqref{e1}$_3$ from $x=a$ to $x$, where $a=0$ or $1$, and then integrate the resulting equation over $\Omega$ to obtain
\begin{equation*}\label{w8}
\begin{split}
\mu\mathbf{w}_x(a,t)=\mu\big(\mathbf{w}^+- \mathbf{w}^-\big)-\int_\Omega(\rho u\mathbf{w}-\mathbf{b})dx-\frac{\partial}{\partial t}\Big(\int_\Omega\int_a^x\rho\mathbf{w}dydx\Big).
\end{split}
\end{equation*}
Taking the inner product  with  $\mathbf{w}(a,t)$ and integrating  over $(0, t)$
yield
\begin{equation*}
\begin{split}
\mu\int_0^t(\mathbf{w}\cdot\mathbf{w}_x)(a,s)ds=&\mu\int_0^t\big(\mathbf{w}^+- \mathbf{w}^-\big)\cdot\mathbf{w}(a,s)ds -\int_0^t\mathbf{w}(a,s)\cdot\Big(\int_\Omega(\rho u\mathbf{w}-\mathbf{b})dx\Big)ds\\
&-\mathbf{w}(a,t)\cdot\Big(\int_\Omega\int_a^x
\rho\mathbf{w}dydx\Big)+\mathbf{w}(a,0)\cdot\Big(\int_\Omega\int_a^x
\rho_0\mathbf{w}_0dydx\Big)\\
&+\int_0^t\mathbf{w}_t(a,t)\cdot\Big(\int_\Omega\int_a^x\rho\mathbf{w}dydx\Big)dt.
\end{split}
\end{equation*}
Using Young inequality and \eqref{ba1}$_1$, we obtain
\begin{equation*}\label{w10}
\begin{split}
\left|\mu\int_0^t(\mathbf{w}\cdot\mathbf{w}_x)(a,s)ds\right|\leq C+\frac12\int_\Omega\mathcal{E}dx+C\iint_{Q_t}\mathcal{E}dxds.
\end{split}
\end{equation*}
Substituting it into \eqref{total} and using  Gronwall inequality, we obtain   \eqref{ba1}$_2$.

\eqref{ba1}$_3$ follows from integrating \eqref{e20}$_2$ and using \eqref{ba1}$_2$. And this completes the proof of the lemma.
\end{proof}

The following   estimates can be obtained as in \cite{FHL,FanJiang,HJ}. For the completeness of the paper, we briefly present its proof.

\begin{lemma}\label{2.2}
Under the assumptions in Theorem \ref{existencethm}, we have
\begin{equation}\label{rho11}
\begin{split}
&  C^{-1}\leq \rho \leq C, \quad \theta\geq C,\\
 &\int_0^T\|\theta\|_{L^\infty(\Omega)}^{q+1-\alpha}dt+\iint_{Q_T} \frac{\kappa\theta_x^2}{\theta^{1+\alpha}} dxdt
\leq C,\quad\forall \alpha \in (0, \min\{1,q\}),\\
&\int_0^T\|\mathbf{b}\|_{L^\infty(\Omega)}^{2}dt+\iint_{Q_T} \left(\lambda
u_x^2+\mu|\mathbf{w}_x|^2+\nu|\mathbf{b}_x|^2 \right)dxdt
\leq C,\\
&\iint_{Q_T} |\theta_x|^{3/2} dxdt\leq C.
\end{split}
\end{equation}
\end{lemma}
\begin{proof}
We first prove that $\rho \leq C.$ Denote
$$
\phi=\int_0^t \tilde{P}(x,s)ds+\int_0^x\rho_0(y)u_0(y)dy,\quad \tilde{P}=\lambda u_x-\rho u^2-\gamma \rho\theta-\frac12|\mathbf{b}|^2.
$$
Then
$$
(\rho u)_t=\tilde{P}_x,~~\phi_t=\tilde{P},~~\phi_x|_{x=0,1}=0,~~\phi|_{t=0}=\int_0^x\rho_0(y)u_0(y)dy.
$$
By Lemma \ref{2.1}, we have that $\|\phi_x\|_{L^\infty(0,T;L^1(\Omega))}+|\int_\Omega\phi dx|\leq C$, thus, $\|\phi\|_{L^\infty(0,T;L^\infty(\Omega))} \leq C$. From this and the fact  that the function $F:=e^{\phi/\lambda}$ satisfies
$$
D_t(\rho F):=\partial_t (\rho F)+u\partial_x (\rho F)=-\frac{1}{\lambda}\left(p+\frac12|\mathbf{b}|^2\right)\rho F\leq0,
$$
it follows that $\rho \leq C.$

  It follows from  \eqref{e1}$_5$
 that
\begin{equation*}
 \begin{split}
\theta_t+u\theta_x-\frac{1}{\rho} (\kappa \theta_x)_x
\geq & \frac{\lambda}{\rho} \left(u_x^2 -\frac{p}{\lambda} u_x \right)
=   \frac{\lambda}{\rho} \left(u_x
-\frac{p}{2\lambda}\right)^2-\frac{\gamma^2}{4\lambda}\rho\theta^2.
\end{split}
\end{equation*}
By $ \rho \leq C$, we have that $
\theta_t+u\theta_x-\frac{1}{\rho} (\kappa
\theta_x)_x   +K \theta^2\geq 0,$
where $K$ is a positive constant independent of $\mu$. Let $z=\theta-\underline\theta$, where
$\underline\theta=\frac{\min_{\overline\Omega}\theta_0}{Ct+1}$
 with $C=K\min_{\overline\Omega}\theta_0$. Then
 $ z_x|_{x=0, 1}=0, ~z|_{t=0}\geq0, $ and
 \begin{equation*}\begin{split}
&z_t+u z_x-\frac{1}{\rho} (\kappa z_x)_x  +K(\theta+\underline\theta)z\\
&=\theta_t+C\frac{\min_{\overline\Omega}\theta_0}{(Ct+1)^2}+u\theta_x-\frac{1}{\rho} (\kappa
\theta_x)_x  +K \theta^2- K\left(\frac{\min_{\overline\Omega}\theta_0}{Ct+1}\right)^2 \geq   0,
\end{split}\end{equation*}
thus,  $z\geq 0$ on
$\overline Q_T$ by the Comparison Theorem that gives $\theta\geq C$.

Multiplying \eqref{e1}$_{5}$ by $\theta^{-\alpha}$ with $\alpha \in (0, \min\{1,q\})$ and integrating over $Q_T$, we have
\begin{equation}\label{equ8}
\begin{split}
 \iint_{Q_T}
 \frac{\lambda u_x^2}{\theta^\alpha} dxdt+\alpha\iint_{Q_T} \frac{\kappa\theta_x^2}{\theta^{1+\alpha}} dxdt\leq \iint_{Q_T}[(\rho\theta)_t+(\rho u\theta)_x+p u_x]\theta^{-\alpha} dxdt.
\end{split}
\end{equation}
From Lemma \ref{2.1}, $\theta\geq C$ and H\"{o}lder  inequality, we obtain
\begin{equation*}
\begin{split}
  \iint_{Q_T}[(\rho\theta)_t+(\rho u\theta)_x]\theta^{-\alpha} dxdt=&\frac{1}{1-\alpha}\int_\Omega \rho\theta^{1-\alpha}dx-\frac{1}{1-\alpha}\int_\Omega \rho_0\theta_0^{1-\alpha}dx
  \leq  C,
\end{split}
\end{equation*}
and
\begin{equation*}
\begin{split}
  \iint_{Q_T} p u_x \theta^{-\alpha} dxdt\leq & \frac12\iint_{Q_T} \frac{u_x^2}{\theta^\alpha}dxdt+C\iint_{Q_T} \rho^2\theta^{2-\alpha}dxdt\\
  \leq & \frac12\iint_{Q_T} \frac{u_x^2}{\theta^\alpha}dxdt+C \int_0^T\|\theta\|^{1-\alpha}_{L^\infty(\Omega)}ds.
\end{split}
\end{equation*}
By the embedding theorem, Young inequality and $\theta\geq C$, we have that if $q\geq 1-\alpha$, then
\begin{equation*}
\begin{split}
  \int_0^T\|\theta\|^{1-\alpha}_{L^\infty(\Omega)}ds \leq & C+C\iint_{Q_T}|\theta^{-\alpha}\theta_x|dxds\\
  \leq & C +C\int_0^T\left(\int_\Omega \frac{|\theta_x|^2\theta^{1-\alpha}}{\theta^{1+\alpha}}dx\right)^{1/2}dt\\
  \leq & \frac{C}{\epsilon} +\epsilon\iint_{Q_T} \frac{\kappa\theta_x^2}{\theta^{1+\alpha}}dxdt,~~\forall \epsilon \in (0,1).
\end{split}
\end{equation*}
If  $0<q<1-\alpha$, then
\begin{equation*}
\begin{split}
  \int_0^T\|\theta\|^{1-\alpha}_{L^\infty(\Omega)}dt \leq & C+C\iint_{Q_T}|\theta^{-\alpha}\theta_x|dxdt\\
  \leq & C +C\int_{0}^T\left(\int_\Omega \frac{\theta^q|\theta_x|^2}{\theta^{1+\alpha}}\theta^{1-\alpha-q}dx\right)^{1/2}dt\\
  \leq & C +\epsilon\iint_{Q_T} \frac{\kappa\theta_x^2}{\theta^{1+\alpha}}dxdt+\frac{C}{\epsilon}\int_0^T\|\theta\|^{1-\alpha-q}_{L^\infty(\Omega)}dt\\
  \leq & C(\epsilon) +\epsilon\iint_{Q_T} \frac{\kappa\theta_x^2}{\theta^{1+\alpha}}dxdt+\frac12\int_0^T\|\theta\|^{1-\alpha}_{L^\infty(\Omega)}dt.
\end{split}
\end{equation*}
Substituting them into \eqref{equ8} and taking a small $\epsilon$, we obtain that $\iint_{Q_T} \frac{\kappa\theta_x^2}{\theta^{1+\alpha}} dxdt\leq C.$ As a consequence, $\int_0^T\|\theta\|_{L^\infty(\Omega)}^{q+1-\alpha}dt\leq C$.

Integrating \eqref{e1}$_{5}$ over $Q_T$ and using Lemma \ref{2.1}, we have
\begin{equation*}
\begin{split}
\iint_{Q_T} \left(\lambda
u_x^2+\mu|\mathbf{w}_x|^2+\nu|\mathbf{b}_x|^2 \right)dxdt=&\frac12\int_\Omega\rho\theta dx-\frac12\int_\Omega\rho_0\theta_0 dx
+\iint_{Q_T}pu_xdxdt\\
\leq &C+\frac{\lambda}{2}\iint_{Q_T} u_x^2dxdt+C\int_0^T\|\theta\|_{L^\infty(\Omega)}\int_\Omega\rho\theta dxdt\\
\leq &C+\frac{\lambda}{2}\iint_{Q_T} u_x^2dxdt.
\end{split}
\end{equation*}
Hence, $\iint_{Q_T} \left(\lambda
u_x^2+\mu|\mathbf{w}_x|^2+\nu|\mathbf{b}_x|^2 \right)dxdt \leq C$. Consequently,  $\int_0^T\|\mathbf{b}\|_{L^\infty(\Omega)}^{2}dt \leq C$.

Simple calculation yields
$$
D_t\left(\frac{1}{\rho F}\right) = \frac{1}{\lambda}\left(p+\frac12|\mathbf{b}|^2\right)\frac{1}{\rho F}.
$$
Thus,  $ \|(\rho F)^{-1} \|_{L^\infty(Q_T)}\leq C$ by using $\int_0^T\big(\|\theta\|_{L^\infty(\Omega)}+\|\mathbf{b}\|_{L^\infty(\Omega)}^{2}\big)dt \leq C$, so $\rho\geq C$.

It remains to show \eqref{rho11}$_4$. By  \eqref{rho11}$_1$ and  \eqref{rho11}$_2$, we  have
\begin{equation}\label{theta01}
\begin{split}
\iint_{Q_T}\frac{\theta_x^2}{\theta} dxdt\leq C.
\end{split}
\end{equation}
Then, from  Lemma \ref{2.1}, \eqref{rho11}$_1$ and
the H\"{o}lder  inequality, we have
  \begin{equation*}
\begin{split}
\theta \leq  \int_\Omega \theta dx
+ \int_\Omega |\theta_x|dx
\leq C+C\left(\int_\Omega\frac{\theta_x^2}{\theta}dx\right)^{1/2}\left(\int_\Omega
\theta  dx\right)^{1/2}.
\end{split}
\end{equation*}
Thus, \eqref{theta01} yields
\begin{equation}\label{theta00}
\begin{split}
 \int_0^T\|\theta\|_{L^\infty(\Omega)}^2dt \leq C.\\
\end{split}
\end{equation}
It follows from the H\"{o}lder inequality, Lemma 2.1 and \eqref{theta00}  that
\begin{equation*}
\begin{split}
\iint_{Q_T}|\theta_x|^{3/2}dxdt\leq&\left(\iint_{Q_T}\frac{\theta_x^2}{\theta} dxdt\right)^{3/4}
\left(\iint_{Q_T}\theta^3 dxdt\right)^{1/4}\\
\leq& C\left(\int_0^T\|\theta^{2}\|_{L^\infty(\Omega)}\int_\Omega \theta dxdt\right)^{1/4}\leq C.\\
\end{split}
\end{equation*}
And this completes the proof of the lemma.
\end{proof}

About the estimation on the magnetic field $\mathbf{b}$, we have

\begin{lemma}\label{2.3}Under the assumptions in Theorem \ref{existencethm}, we have
\begin{equation*}\label{b00}
\begin{split}
& \sup\limits_{0<t<T}\int_\Omega|\mathbf{b}|^4dx+\iint_{Q_T}|\mathbf{b}|^2|\mathbf{b}_x|^2dxdt\leq C.
\end{split}
\end{equation*}
\end{lemma}
\begin{proof}Taking the inner product of \eqref{e1}$_4$ with $4|\mathbf{b}|^2\mathbf{b}$ and integrating over $Q_t$, we obtain
\begin{equation}\label{b0}
\begin{split}
&\int_\Omega|\mathbf{b}|^4dx+4\nu\iint_{Q_t}|\mathbf{b}|^2|\mathbf{b}_x|^2dxds+8\nu\iint_{Q_t}|\mathbf{b}\cdot \mathbf{b}_x|^2dxds\\
&=\int_\Omega|\mathbf{b}_0|^4dx+4\iint_{Q_t} \mathbf{w}_x \cdot(|\mathbf{b}|^2\mathbf{b})dxds-4\iint_{Q_t} (u \mathbf{b})_x\cdot(|\mathbf{b}|^2\mathbf{b})dxds.\\
\end{split}
\end{equation}
Using the  Young inequality, we have
\begin{equation}\label{b1}
\begin{split}
  \iint_{Q_t} \mathbf{w}_x \cdot(\mathbf{b} |\mathbf{b}|^2) dxds
 &=-\iint_{Q_t} \mathbf{w} \cdot(\mathbf{b}_x |\mathbf{b}|^2)dxds-2\iint_{Q_t} (\mathbf{w} \cdot\mathbf{b} )(\mathbf{b}\cdot \mathbf{b}_x)dxds\\
 &\leq \frac{\nu}{4}\iint_{Q_t} |\mathbf{b}|^2 |\mathbf{b}_x|^2dxds+C\iint_{Q_t} |\mathbf{w}|^2 |\mathbf{b}|^2 dxds\\
 &\leq \frac{\nu}{4}\iint_{Q_t} |\mathbf{b}|^2 |\mathbf{b}_x|^2dxds+C\int_0^t\|\mathbf{b}\|_{L^\infty(\Omega)}^2 \int_\Omega |\mathbf{w}|^2dxds\\
 &\leq\frac{\nu}{4}\iint_{Q_t} |\mathbf{b}|^2 |\mathbf{b}_x|^2dxds+C,
\end{split}
\end{equation}
where we have used \eqref{ba1}$_2$ and \eqref{rho11}$_3$. On the other hand, we have
\begin{equation}\label{b2}
\begin{split}
 &-\iint_{Q_t} (u \mathbf{b})_x\cdot|\mathbf{b}|^2\mathbf{b}dxds=3\iint_{Q_t}  u  (\mathbf{b}_x \cdot\mathbf{b})|\mathbf{b}|^2 dxds\\
 &\leq  \frac{\nu}{4}\iint_{Q_t} |\mathbf{b}|^2  |\mathbf{b}_x|^2 dxds+C\iint_{Q_t}   u^2 |\mathbf{b}|^4 dxds\\
 &\leq\frac{\nu}{4}\iint_{Q_t} |\mathbf{b}|^2  |\mathbf{b}_x|^2 dxds+C\int_0^t\|u^2\|_{L^\infty(\Omega)}\int_\Omega |\mathbf{b}|^4 dxds.
\end{split}
\end{equation}
Plugging \eqref{b1} and \eqref{b2} into \eqref{b0} and using the
Gronwall inequality, we complete the proof of the lemma by noticing
$\int_0^T\|u^2\|_{L^\infty(\Omega)} dt\leq C\iint_{Q_T}u_x^2dxdt\leq
C$.
\end{proof}

\begin{lemma}\label{2.4}Under the assumptions in Theorem \ref{existencethm}, we have
\begin{equation}\label{rho}
\begin{split}
&\sup\limits_{0<t<T}\int_\Omega\rho_x^2dx+\iint_{Q_T} \big(\rho_t^2+\theta
\rho_x^2\big) dxdt\leq C,\\
&\left|\rho(x,t)-\rho(y,s)\right|\leq C\left(|x-y|^{1/2}
+|s-t|^{1/4}\right),\quad\forall (x, t), (y, s) \in \overline Q_T.
\end{split}
\end{equation}
\end{lemma}

\begin{proof}
Set $\eta=1/\rho$. It follows from the equation \eqref{e1}$_1$ that
$u_x=\rho(\eta_t+u\eta_x).$
Substituting it into \eqref{e1}$_2$ yields
\begin{equation*}
\left[\rho(u-\lambda\eta_x)\right]_t+\left[\rho
u(u-\lambda\eta_x)\right]_x
=\gamma\rho^2(\theta\eta_x-\eta\theta_x)-\mathbf{b}\cdot\mathbf{b}_x.
\end{equation*}
Multiplying it by $(u-\lambda\eta_x)$ and integrating over $Q_t$, we
have
\begin{equation}\label{equ10}
\begin{split}
&\frac12\int_\Omega\rho
(u-\lambda\eta_x)^2dx+\gamma\lambda\iint_{Q_t}
\theta\rho^2\eta_x^{2}dxds\\
&=\frac12\int_\Omega\rho_0(u_0+ \lambda\rho_0^{-2}\rho_{0x})^2dx+\gamma\iint_{Q_t}\rho^2\theta u  \eta_x dxds
\\
&\quad-\gamma\iint_{Q_t}\rho^2\eta\theta_x (u-\lambda\eta_x)dxds-\iint_{Q_t}\mathbf{b}\cdot\mathbf{b}_x(u-\lambda\eta_x)dxds.
\end{split}
\end{equation}
Using the Young inequality and  Lemmas 2.1-2.2, we  obtain
\begin{equation}\label{equ11}
\begin{split}
 \gamma\iint_{Q_t}\rho^2\theta u  \eta_x dxds
 &\leq
\frac{\gamma\lambda}{2}\iint_{Q_t}\theta\rho^{2}\eta_x^2dxds+C\iint_{Q_t} \theta u^2 dxds\\
&\leq  \frac{\gamma\lambda}{2}\iint_{Q_t}\theta\rho^{2}\eta_x^2dxds+C\int_0^t\|\theta\|_{L^\infty(\Omega)}
\int_{\Omega} u^2 dxds\\
&\leq
C+\frac{\gamma\lambda}{2}\iint_{Q_t}\theta\rho^{2}\eta_x^2dxds.
\end{split}
\end{equation}
By the Cauchy inequality, \eqref{theta01} and Lemma \ref{2.3}, we have
\begin{equation}\label{equ9}
\begin{split}
&-\gamma\iint_{Q_t}\rho^2\eta\theta_x (u-\lambda\eta_x)dxds-\iint_{Q_t}\mathbf{b}\cdot\mathbf{b}_x(u-\lambda\eta_x)dxds\\
&\leq  C+C\iint_{Q_t}\theta \rho
(u-\lambda\eta_x)^2dxds+C\iint_{Q_t}\frac{\theta_x^2}{\theta}dxds+C\iint_{Q_t} \rho(u-\lambda\eta_x)^2dxds\\
&\leq  C+C\int_0^t\big(1+\|\theta\|_{L^\infty(\Omega)}\big)\int_\Omega\rho
(u-\lambda\eta_x)^2dxds.
\end{split}
\end{equation}
Substituting \eqref{equ11} and \eqref{equ9} into \eqref{equ10} and using the Gronwall inequality, we obtain
 \begin{equation*}\label{rho22}
\begin{split}
&\sup\limits_{0<t<T}\int_\Omega\rho_x^2dx+\iint_{Q_T}  \theta \rho_x^2 dxds\leq C .
\end{split}
\end{equation*}
 Using this estimate  and Lemma 2.2, one can derive from the equation \eqref{e1}$_1$ that
\begin{equation*}
\begin{split}
 \iint_{Q_T} \rho_t^2dxdt \leq & C\int_0^T\|u^2\|_{L^\infty(\Omega)}\int_\Omega\rho_x^2dxdt+C\iint_{Q_T}u_x^2dxdt \leq C,\\
\end{split}
\end{equation*}
 that implies \eqref{rho}$_1$.

  We now turn to \eqref{rho}$_2$. Let $\beta(x)=\rho(x,t)-\rho(x,s)$ for any $x \in [0, 1]$
  and $s, t \in [0, T]$ with $s\neq t$. Then for any $x\in [0, 1]$ and $\delta \in (0, 1/2]$, there exist  some
  $y\in [0, 1]$ and $\xi$ between $x$ and $y$ such that $\delta=|y-x|$  and
  $\beta(\xi)=\frac{1}{x-y}\int^x_y\beta(z)dz$, and
\begin{equation*}
\begin{split}
\beta(x)=\frac{1}{x-y}\int^x_y\beta(z)dz+\int_\xi^x\beta'(z)dz.
\end{split}
\end{equation*}
Thus, from the H\"{o}lder inequality and \eqref{rho}$_1$, we have
\begin{equation*}
\begin{split}
|\beta(x)|\leq& \frac{1}{\delta}\left|\int^x_y\beta(z)dz\right|+\left|\int_\xi^x\beta'(z)dz\right|\\
\leq &\frac{1}{\delta}\left|\int^x_y\hspace{-2mm}\int_s^t\rho_\tau d\tau dz\right|+\left|\int_\xi^x\left[\rho_z(z,t)-\rho_z(z,s)\right]dz\right|\\
\leq &\frac{1}{\delta}\left(\iint_{Q_T}\rho_\tau^2 d\tau dz\right)^{1/2}|x-y|^{1/2}|s-t|^{1/2} + \left(2\sup\limits_{0<t<T}\int_0^1 |\rho_z(z,t)|^2 dz\right)^{1/2}|x-\xi|^{1/2}\\
\leq & C\delta^{-1/2}|s-t|^{1/2}+C\delta^{1/2}.
\end{split}
\end{equation*}
If $0<|s-t|^{1/2}<1/2$, taking $\delta=|s-t|^{1/2}$ yields
\begin{equation}\label{rho6}
\begin{split}
|\rho(x,s)-\rho(x,t)|\leq C|s-t|^{1/4}.
\end{split}
\end{equation}
If $ |s-t|^{1/2}\geq 1/2$, then \eqref{rho6} holds because $\rho$ is uniformly bounded in $\mu$.

From \eqref{rho}$_1$, we have that
$
|\rho(x,t)-\rho(y,t)|=\Big|\int_y^x\rho_zdz\Big|\leq C|x-y|^{1/2}.
$
Thus, \eqref{rho}$_2$ is proved and this completes the proof of the lemma.
\end{proof}

The following  is the key lemma in this paper that leads to a new approach for
the estimation on the uniform bounds in the general setting presented in
 this paper.

\begin{lemma}\label{2.5} Under the assumptions in Theorem \ref{existencethm}, we have
\begin{equation}\label{uxx}
\begin{split}
\iint_{Q_T}|u_{xx}|^{m_0} dxdt\leq C,\quad m_0=\min\{m, 4/3\}.
\end{split}
\end{equation}
In particular,
\begin{equation}\label{u0}
\begin{split}
\int_0^T\|u_{x}\|_{L^\infty(\Omega)}^{m_0}dt \leq C.
\end{split}
\end{equation}
\end{lemma}
\begin{proof} We will apply the $L^p$ estimates of linear parabolic equations  (see \cite[Theorem 7.17]{Lie}) to obtain \eqref{uxx}. Firstly, rewrite the equation \eqref{e1}$_2$ as
\begin{equation}\label{f2}
\begin{split}
&
 u_t-\frac{\lambda}{\rho}u_{xx}=-uu_x-
 \gamma\theta_x-\frac{\gamma}{\rho}\rho_x\theta-\frac{1}{\rho}\mathbf{b}\cdot\mathbf{b}_x=:f.
\end{split}
\end{equation}
 From \eqref{rho}$_2$,
the coefficient $a(x,t):=\lambda/\rho$ is uniformly bounded in $C^{1/2,1/4}(\overline Q_T)$. By noticing
  the condition  $u_0 \in W^{2,m}(\Omega)$ for some $m>1$ in \eqref{assumption1}, it suffices to give a uniform bound of $f$ in $L^{4/3}(Q_T)$.

From Lemmas 2.2 and 2.3,  the second term and the forth term on
the  right hand side of \eqref{f2} are
 uniformly bounded in $L^{3/2}(Q_T)$ and $L^{2}(Q_T)$, respectively.

Using the H\"{o}lder inequality and Lemma 2.1, we obtain
\begin{equation*}
\begin{split}
 u^2&\leq2\int_\Omega|uu_x|dx \leq 2\left(\int_\Omega u^2dx\right)^{1/2}\left(\int_\Omega u_x^2dx\right)^{1/2} \leq C\left(\int_\Omega u_x^2dx\right)^{1/2}.
\end{split}
\end{equation*}
It follows from Lemma 2.2  that
$\int_0^T\|u\|_{L^\infty(\Omega)}^4 dt \leq   C.$
The Young inequality gives
\begin{equation*}
\begin{split}
  \iint_{Q_T}|uu_x|^{3/2}dxdt &\leq  C\iint_{Q_T}u_x^2dxdt+C\iint_{Q_T}u^6dxdt\\
 &\leq C+ C \int_0^T\|u\|_{L^\infty(\Omega)}^4 \int_{\Omega}u^2dxdt \leq C.
\end{split}
\end{equation*}

For the third term on the right hand side of \eqref{f2}, we use \eqref{rho}$_1$ and \eqref{theta00} to obtain
\begin{equation*}
\begin{split}
  \iint_{Q_T}|\rho_x \theta|^{4/3}dxdt
 &\leq C\iint_{Q_T}\rho_x^2\theta dxdt+C\iint_{Q_T}\theta^{2}dxdt
 \leq C.
\end{split}
\end{equation*}

Consequently, $\|f\|_{L^{4/3}(Q_T)}\leq C.$ Thus \eqref{uxx} is proved.

 \eqref{u0} is an immediate consequence of \eqref{uxx}, and this completes
 the proof of the lemma.
\end{proof}

As a direct application of Lemma 2.5, we have
\begin{lemma}\label{2.6}
Under the assumptions in Theorem \ref{existencethm}, we have
\begin{equation*}\label{we11}
\begin{split}
& \mu\sup\limits_{0<t<T}\int_\Omega|\mathbf{w}_x|^2dx+\mu^2 \iint_{Q_T}
 |\mathbf{w}_{xx}|^2 dxdt  \leq C.
\end{split}
\end{equation*}
 \end{lemma}
\begin{proof}
Rewrite  \eqref{e1}$_3$ as
\begin{equation}\label{w12}
\begin{split}
 -\mathbf{w}_t+\frac{\mu}{\rho}\mathbf{w}_{xx}=u\mathbf{w}_x-\frac{1}{\rho}\mathbf{b}_x.
\end{split}
\end{equation}
Taking the inner product  with $\mu\mathbf{w}_{xx}$ and integrating over $Q_t$
yield
\begin{equation}\label{0v100}
\begin{split}
 & \frac\mu2\int_\Omega|\mathbf{w}_x|^2dx
 +\mu^2\iint_{Q_t}\frac{1}{\rho}|\mathbf{w}_{xx}|^2dxdx\\
  &=\frac\mu2\int_\Omega|\mathbf{w}_{0x}|^2dx-\mu\iint_{Q_t}\frac{1}{\rho}\mathbf{b}_x\cdot\mathbf{w}_{xx}dxds\\
&\quad -\frac{\mu}{2}\iint_{Q_t}u_x |\mathbf{w}_x|^2dxds+\mu\int_0^t \mathbf{w}_{t}\cdot\mathbf{w}_{x} \Big|_{x=0}^{x=1} ds\\
& \leq C\mu+\frac{\mu^2}{4}\iint_{Q_t}\frac{1}{\rho}|\mathbf{w}_{xx}|^2dxds+C\iint_{Q_t}|\mathbf{b}_x|^2dxds\\
&\quad+ C\int_0^t\|u_x\|_{L^\infty(\Omega)}\left(\mu\int_\Omega|\mathbf{w}_x|^2dx\right)ds+C\mu\int_0^t\|\mathbf{w}_x\|_{L^\infty(\Omega)} ds.
  \end{split}
\end{equation}
By the embedding  theorem and  H\"older inequality, we obtain
\begin{equation}\label{wx2}
\begin{split}
   |\mathbf{w}_x|^2\leq & C\Big(\int_\Omega|\mathbf{w}_x|^2dx+\int_\Omega|\mathbf{w}_{x}||\mathbf{w}_{xx}|dx\Big)\\
   \leq& C\int_\Omega|\mathbf{w}_x|^2dx+C\left(\int_\Omega|\mathbf{w}_{x}|^2dx\right)^{1/2}\left(\int_\Omega|\mathbf{w}_{xx}|^2dx\right)^{1/2}.
\end{split}
\end{equation}
Thus, using  the Young inequality yields
\begin{equation*}\label{w5}
\begin{split}
   &\mu\int_0^t\|\mathbf{w}_x\|_{L^\infty(\Omega)} ds\\
   &\leq  C\mu\int_0^t\left(\int_\Omega
|\mathbf{w}_{x}|^2dx\right)^{1/2}ds+C\int_0^t\mu^{1/4}\left(\mu\int_\Omega|\mathbf{w}_{x}|^2dx\right)^{1/4}
   \left(\mu^2\int_\Omega|\mathbf{w}_{xx}|^2dx\right)^{1/4}ds\\[1mm]
&\leq  C\sqrt{\mu}+\frac{C  \mu}{\epsilon}\iint_{Q_t}
|\mathbf{w}_{x}|^2dxds+\epsilon \mu^2
\iint_{Q_t}\frac{1}{\rho}|\mathbf{w}_{xx}|^2dxds,~~\forall \epsilon \in (0, 1).
\end{split}
\end{equation*}
Plugging  it into \eqref{0v100} and taking a small $\epsilon>0$, we have
\begin{equation*}
\begin{split}
  & \mu\int_\Omega|\mathbf{w}_x|^2dx
 +\mu^2\iint_{Q_t}\frac{1}{\rho}|\mathbf{w}_{xx}|^2dxds  \leq C  +C\int_0^t\big(1+\|u_x\|_{L^\infty}\big)\left(\mu\int_{\Omega}|\mathbf{w}_x|^2dx\right)ds.
  \end{split}
\end{equation*}
Thus, the lemma follows from the Gronwall inequality and  \eqref{u0}.
\end{proof}

We now turn to prove the other estimates  in
Theorem 1.1. For this, we need the following three lemmas.

\begin{lemma}\label{2.7} Under the assumptions in Theorem \ref{existencethm}, we have
\begin{equation*}
\begin{split}
    & \int_\Omega|\mathbf{b}_x|^2\omega^2dx+ \iint_{Q_t}|\mathbf{b}_{xx}|^2\omega^2dxds
    \leq C\iint_{Q_t}  |\mathbf{w}_{x}|^2\omega^2  dxds+C\left(\iint_{Q_t} u_{xx}^2
    dxds\right)^{1/2},
  \end{split}
\end{equation*}
where $\omega$ is the same as the one defined in Theorem 1.1.
 \end{lemma}
 \begin{proof}
Taking the inner product of  \eqref{e1}$_4$ with $\mathbf{b}_{xx}\omega^2(x)$ and
integrating over $Q_t$ give
\begin{equation}\label{b111}
\begin{split}
   &-\iint_{Q_t} \mathbf{b}_t\cdot \mathbf{b}_{xx}\omega^2 dxdt+\nu\iint_{Q_t} |\mathbf{b}_{xx}|^2\omega^2 dxds\\
    &=\iint_{Q_t}  (u\mathbf{b})_x\cdot \mathbf{b}_{xx}\omega^2  dxds-\iint_{Q_t} \mathbf{w}_x\cdot \mathbf{b}_{xx}\omega^2 dxds.
  \end{split}
\end{equation}
To estimate the first integral on the left hand side of \eqref{b111}, we use
integration by parts and  \eqref{e1}$_4$ to obtain
\begin{equation}\label{be1}
\begin{split}
   \iint_{Q_t} \mathbf{b}_t\cdot \mathbf{b}_{xx}\omega^2 dxds
   &=-\frac{1}{2}\int_\Omega|\mathbf{b}_x|^2\omega^2   dx
   +\frac{1}{2}\int_\Omega|\mathbf{b}_{0x}|^2\omega^2   dx-2\iint_{Q_t} \mathbf{b}_t\cdot \mathbf{b}_{x}\omega\omega' dxds\\
   &=-\frac{1}{2}\int_\Omega|\mathbf{b}_x|^2\omega^2dx
   +\frac{1}{2}\int_\Omega|\mathbf{b}_{0x}|^2\omega^2dx\\
    &\quad-2\iint_{Q_t} \left(\nu\mathbf{b}_{xx}+\mathbf{w}_x-u\mathbf{b}_x-u_x\mathbf{b}\right)\cdot \mathbf{b}_{x}\omega\omega'  dxds.
  \end{split}
\end{equation}
For the third term on right hand side of
\eqref{be1}, by the Cauchy inequality and Lemmas \ref{2.2} and \ref{2.3}, we have
\begin{equation*}
\begin{split}
    &-2\iint_{Q_t} \left(\nu\mathbf{b}_{xx}+\mathbf{w}_x-u_x\mathbf{b}\right)\cdot \mathbf{b}_{x}\omega\omega'  dxds\\
   &\leq  \frac{\nu}{4}\iint_{Q_t}|\mathbf{b}_{xx}|^2\omega^2dxds  +C\iint_{Q_t}|\mathbf{b}_{x}|^2dxds+C\iint_{Q_t}  |\mathbf{w}_{x}|^2\omega^2  dxds\\
    & \quad+  C\iint_{Q_t} u_x^2dxds+C\iint_{Q_t} |\mathbf{b}\cdot\mathbf{b}_x|^2 dxds\\
    &\leq C+ \frac{\nu}{4}\iint_{Q_t}|\mathbf{b}_{xx}|^2\omega^2dxds  +C\iint_{Q_t}  |\mathbf{w}_{x}|^2\omega^2  dxds.\\
  \end{split}
\end{equation*}
By noticing $u(1,t)=u(0,t)=0$, we have
\begin{equation}\label{u3}
\begin{split}
|u(x,t)|\leq \|u_x\|_{L^\infty(\Omega)}\omega(x),
\end{split}
\end{equation}
thus,
\begin{equation*}
\begin{split}
    &2\iint_{Q_t} u |\mathbf{b}_{x}|^2\omega\omega'dxds
    \leq  C\int_{0}^t \|u_x\|_{L^\infty(\Omega)} \int_\Omega|\mathbf{b}_{x}|^2\omega^2dxds.\\
  \end{split}
\end{equation*}
Substituting them into \eqref{be1} yields
\begin{equation}\label{be3}
\begin{split}
     \iint_{Q_t} \mathbf{b}_t\cdot \mathbf{b}_{xx}\omega^2 dxdt  \leq & C-\frac{1}{2}\int_\Omega|\mathbf{b}_x|^2\omega^2dx+\frac{\nu}{4}\iint_{Q_t}|\mathbf{b}_{xx}|^2\omega^2dxds\\
& +C\int_{0}^t\|u_x\|_{L^\infty(\Omega)}\int_\Omega|\mathbf{b}_{x}|^2\omega^2dxds
 +C\iint_{Q_t}  |\mathbf{w}_{x}|^2\omega^2  dxds.
  \end{split}
\end{equation}
As to the two terms on the right hand side of \eqref{b111}, we use the  Young inequality to obtain
\begin{equation}\label{be2}
\begin{split}
   & \iint_{Q_t}  (u\mathbf{b})_x\cdot \mathbf{b}_{xx}\omega^2  dxds-\iint_{Q_t} \mathbf{w}_x\cdot \mathbf{b}_{xx}\omega^2 dxds\\
   &\leq  \frac{\nu}{4}\iint_{Q_t} |\mathbf{b}_{xx}|^2\omega^2  dxds+C\iint_{Q_t}  |(u\mathbf{b})_x|^2\omega^2  dxds
   +C\iint_{Q_t}  |\mathbf{w}_{x}|^2\omega^2  dxds.
  \end{split}
\end{equation}
It remains to treat the second term on the right hand side of \eqref{be2}.  By Lemma \ref{2.2}, we obtain
\begin{equation}\label{ux}
\begin{split}
     &\int_0^t\|u_x\|_{L^\infty(\Omega)}^2 ds\leq C\iint_{Q_t}|u_xu_{xx}|dxds  \leq C\left(\iint_{Q_t}u_{xx}^2 dxds\right)^{1/2}.
  \end{split}
\end{equation}
It follows from Lemma 2.1 that
\begin{equation*}
\begin{split}
    \iint_{Q_t}  |(u\mathbf{b})_x|^2\omega^2  dxds \leq &C\iint_{Q_t} u^2|\mathbf{b}_{x}|^2\omega^2  dxds+C\iint_{Q_t} u_x^2|\mathbf{b} |^2\omega^2  dxds\\
  \leq & C\int_0^t\|u^2\|_{L^\infty(\Omega)}\int_\Omega|\mathbf{b}_{x}|^2\omega^2  dxds+C\left(\iint_{Q_t} u_{xx}^2 dxds\right)^{1/2}.
  \end{split}
\end{equation*}
Substituting it into \eqref{be2} and then  substituting the resulting inequality and \eqref{be3} into \eqref{b111}, then
the  Gronwall inequality completes the proof of the lemma.
\end{proof}

\begin{lemma}\label{2.8}Under the assumptions in Theorem \ref{existencethm}, we have
\begin{equation}\label{w1}
\begin{split}
   &  \int_\Omega|\mathbf{w}_x|^2\omega^2    dx+\mu \iint_{Q_t} |\mathbf{w}_{xx}|^2\omega^2 dxds\leq C+C\left(\iint_{Q_t} u_{xx}^2 dxds\right)^{1/2},\\
   &\iint_{Q_t}\big(|\mathbf{w}_t|^2+u^2|\mathbf{w}_x|^2\big)dxdt\leq C+C \iint_{Q_t}u_{xx}^2 dxds.
  \end{split}
\end{equation}
\end{lemma}
\begin{proof}
  Taking the inner product of \eqref{w12} with $\mathbf{w}_{xx}\omega^2(x)$ and integrating over $Q_t$ give
\begin{equation}\label{0v10}
\begin{split}
   &-\iint_{Q_t} \mathbf{w}_t\cdot \mathbf{w}_{xx}\omega^2 dxdt+\mu\iint_{Q_t}|\mathbf{w}_{xx}|^2\frac{\omega^2}{\rho}dxds\\
    &=\iint_{Q_t} u\mathbf{w}_x\cdot \mathbf{w}_{xx}\omega^2 dxds-\iint_{Q_t}  \mathbf{b}_x\cdot \mathbf{w}_{xx}\frac{\omega^2}{\rho}  dxds.
  \end{split}
\end{equation}
Then \eqref{w12} gives
\begin{equation*}
\begin{split}
   & \iint_{Q_t} \mathbf{w}_t\cdot \mathbf{w}_{xx}\omega^2 dxdt\\[1mm]
   &=-\frac{1}{2}\int_\Omega|\mathbf{w}_x|^2\omega^2   dx
  +\frac{1}{2}\int_\Omega|\mathbf{w}_{0x}|^2\omega^2   dx-2\iint_{Q_t} \mathbf{w}_t\cdot \mathbf{w}_{x}\omega\omega' dxdt\\[1mm]
   &=-\frac{1}{2}\int_\Omega|\mathbf{w}_x|^2\omega^2  dx
   +\frac{1}{2}\int_\Omega|\mathbf{w}_{0x}|^2\omega^2   dx -2\iint_{Q_t} \left(\frac{\mu}{\rho}\mathbf{w}_{xx}-u\mathbf{w}_x
   +\frac{\mathbf{b}_x}{\rho}\right)\cdot \mathbf{w}_{x}\omega\omega'  dxds\\[1mm]
   &\leq C-\frac{1}{2}\int_\Omega|\mathbf{w}_x|^2\omega^2   dx+C\mu^2\iint_{Q_t}|\mathbf{w}_{xx}|^2dxds +C\iint_{Q_t}|\mathbf{w}_x|^2\omega^2   dxds\\[1mm]
   &\quad+C\iint_{Q_t}|\mathbf{b}_{x}|^2dxds+C\iint_{Q_t}|u||\mathbf{w}_x|^2\omega dxds.
  \end{split}
\end{equation*}
From \eqref{u3}, we have
\begin{equation*}
\begin{split}
    \iint_{Q_t}|u||\mathbf{w}_x|^2\omega dxds\leq \int_0^t\|u_x\|_{L^\infty(\Omega)} \int_\Omega|\mathbf{w}_x|^2\omega^2dxds.\\
  \end{split}
\end{equation*}
Thus, from Lemmas \ref{2.2} and \ref{2.6}, we obtain
\begin{equation*}
\begin{split}
    \iint_{Q_t} \mathbf{w}_t\cdot \mathbf{w}_{xx}\omega^2 dxdt  \leq C-\frac{1}{2}\int_\Omega|\mathbf{w}_x|^2\omega^2   dx
   +C\int_0^t\big(1+\|u_x\|_{L^\infty(\Omega)} \big)\int_\Omega|\mathbf{w}_x|^2\omega^2dxds.\\
  \end{split}
\end{equation*}
To estimate the right hand side of \eqref{0v10}, we use  \eqref{u3} to obtain
\begin{equation*}
\begin{split}
  \iint_{Q_t} u\mathbf{w}_x\cdot \mathbf{w}_{xx}\omega^2 dxds
  =&-\frac{1}{2}\iint_{Q_t} |\mathbf{w}_x|^2[u_x\omega^2+2u\omega\omega']
  dxds\\
  \leq&C\int_0^t \|u_x\|_{L^\infty(\Omega)} \int_\Omega|\mathbf{w}_x|^2\omega^2dxds,\\
  \end{split}
\end{equation*}
and
\begin{equation*}
\begin{split}
   -\iint_{Q_t} \mathbf{b}_x\cdot \mathbf{w}_{xx}\frac{\omega^2}{\rho} dxds &=\iint_{Q_t}\Big(\mathbf{w}_x\cdot \mathbf{b}_{xx}  \frac{\omega^2}{\rho}dxds+2 \mathbf{w}_x\cdot \mathbf{b}_{x} \frac{\omega\omega'}{\rho} -  \mathbf{w}_x\cdot \mathbf{b}_{x}\frac{\omega^2\rho_x}{\rho^2}\Big) dxds\\[1mm]
& \leq C\iint_{Q_t} \Big( |\mathbf{b}_{xx}|^2\omega^2 + |\mathbf{w}_x|^2\omega^2 +  |\mathbf{b}_{x}|^2+|\mathbf{b}_x|^2\omega^2 \rho_x^2\Big)dxds\\[1mm]
&\leq  C+C\iint_{Q_t}  |\mathbf{w}_x|^2\omega^2 dxds+C\iint_{Q_t}  |\mathbf{b}_{xx}|^2\omega^2 dxds,
  \end{split}
\end{equation*}
where we  have used the fact that by using Lemmas \ref{2.2} and \ref{2.4}, it holds
\begin{equation*}
\begin{split}
     \iint_{Q_t}  |\mathbf{b}_x|^2\omega^2 \rho_x^2dxds\leq &C\int_0^t  \left\||\mathbf{b}_x|^2\omega^2\right\|_{L^\infty(\Omega)}  ds
      \leq   C\iint_{Q_t}  \left|(|\mathbf{b}_x|^2\omega^2)_x\right| dxds\\
    \leq & C\iint_{Q_t}    |\mathbf{b}_{x}|^2|\omega\omega'|dxds+C \iint_{Q_t}    |\mathbf{b}_x\cdot\mathbf{b}_{xx}|  \omega^2 dxds\\[1mm]
    \leq &C+C\iint_{Q_t}    |\mathbf{b}_{xx}|^2 \omega^2 dxds.
  \end{split}
\end{equation*}
Substituting the above estimates into \eqref{0v10} and using Lemma \ref{2.7}, we have
\begin{equation*}
\begin{split}
     &\int_\Omega|\mathbf{w}_x|^2\omega^2   dx+\mu \iint_{Q_t} |\mathbf{w}_{xx}|^2\omega^2   dxds\\
     & \leq C + C \int_0^t\big(1+\|u_x\|_{L^\infty(\Omega)}\big)\int_\Omega|\mathbf{w}_x|^2\omega^2  dxds+C\left(\iint_{Q_t} u_{xx}^2 dxds\right)^{1/2}.
  \end{split}
\end{equation*}
Thus, the first estimate in the  lemma follows from
the Gronwall inequality and
\eqref{u0}.

Consequently, from  \eqref{u3}, the first estimate of the lemma and \eqref{ux}, we obtain
\begin{equation*}\label{v11}
\begin{split}
    \iint_{Q_T} u^2|\mathbf{w}_x|^2 dxdt&\leq \int_0^T\|u_x\|_{L^\infty(\Omega)}^2\int_\Omega |\mathbf{w}_x|^2\omega^2 dxdt \\ &\leq C\int_0^T\|u_x\|_{L^\infty(\Omega)}^2 dt\left[1+\left(\iint_{Q_T}u_{xx}^2 dxdt\right)^{1/2}\right]\\
    &\leq C+C \iint_{Q_T}u_{xx}^2 dxdt.
  \end{split}
\end{equation*}
Furthermore, by Lemmas \ref{2.2} and \ref{2.6}, we derive from
\eqref{w12} that
\begin{equation*}
\begin{split}
   &\iint_{Q_T} |\mathbf{w}_t|^2dxdt\leq C +C\iint_{Q_T}u^2|\mathbf{w}_x|^2dxdt\leq C+C \iint_{Q_T}u_{xx}^2 dxdt.
  \end{split}
\end{equation*}
And this completes the proof of the lemma.
\end{proof}

\begin{lemma}\label{2.9} Under the assumptions in Theorem \ref{existencethm}, we have
\begin{equation*}\label{b99}
\begin{split}
  &  \int_\Omega|\mathbf{b}_{x}|^2dx+\iint_{Q_t} |\mathbf{b}_t|^2
  dxdt
  \leq C+ C\left(\iint_{Q_t}u_{xx}^2 dxds\right)^{1/2}.  \\
  \end{split}
\end{equation*}
\end{lemma}
\begin{proof}
Taking the inner product of \eqref{e1}$_4$ with $\mathbf{b}_t$ and
integrating over $Q_t$ yield
\begin{equation}\label{b9}
\begin{split}
  &\frac{\nu}{2}\int_\Omega|\mathbf{b}_{x}|^2dx+ \iint_{Q_t} |\mathbf{b}_t|^2
  dxdt
  =\frac{\nu}{2}\int_\Omega|\mathbf{b}_{0x}|^2dx+\iint_{Q_t}  \big[\mathbf{w}_x
  - (u\mathbf{b})_{x}\big]\cdot\mathbf{b}_t dxdt.  \\
  \end{split}
\end{equation}
Note that
\begin{equation*}
\begin{split}
   \iint_{Q_t}   \mathbf{w}_x  \cdot\mathbf{b}_t dxdt&= \int_\Omega\mathbf{w}_x  \cdot\mathbf{b}dx
  -\int_\Omega\mathbf{w}_{0x}  \cdot\mathbf{b}_0dx-  \iint_{Q_t}   (\mathbf{w}_t)_x  \cdot\mathbf{b}  dxdt\\
  &=-\int_\Omega\mathbf{w}_{0x}\cdot\mathbf{b}_0dx -\int_\Omega\mathbf{w}\cdot\mathbf{b}_xdx
  +\iint_{Q_t}   \mathbf{w}_t\cdot\mathbf{b}_x  dxdt.
  \end{split}
\end{equation*}
Thus, from  Lemmas \ref{2.1} and \ref{2.2} and
\eqref{w1}$_2$, we obtain
\begin{equation*}\label{b10}
\begin{split}
   \iint_{Q_t}   \mathbf{w}_x  \cdot\mathbf{b}_t dxdt
  &  \leq C+\frac{\nu}{4}\int_\Omega|\mathbf{b}_x|^2dx+\left(\iint_{Q_t}|\mathbf{b}_x|^2dxdt\right)^{1/2}
\left(\iint_{Q_t}|\mathbf{w}_t|^2dxdt\right)^{1/2}\\
&\leq C +\frac{\nu}{4}\int_\Omega|\mathbf{b}_x|^2dx+C\left(\iint_{Q_t}|\mathbf{w}_t|^2dxdt\right)^{1/2}\\
&\leq C +\frac{\nu}{4}\int_\Omega|\mathbf{b}_x|^2dx+C\left(\iint_{Q_t}u_{xx}^2dxds\right)^{1/2}.
  \end{split}
\end{equation*}
From the Cauchy inequality, Lemma 2.1 and \eqref{ux}, we have
\begin{equation*}\label{b11}
\begin{split}
   &-\iint_{Q_t}  (u\mathbf{b})_{x} \cdot\mathbf{b}_t dxdt\\
   & \leq \frac12  \iint_{Q_t} |\mathbf{b}_t|^2 dxdt
  +\frac12\iint_{Q_t}  |(u\mathbf{b})_{x}|^2 dxds\\
   &\leq   \frac12  \iint_{Q_t} |\mathbf{b}_t|^2 dxdt+C\int_0^t\|u^2\|_{L^\infty(\Omega)} \int_{\Omega}  |\mathbf{b}_x|^2 dxds+C\int_0^t\|u_x\|_{L^\infty(\Omega)}^2\int_{\Omega}  |\mathbf{b}|^2 dxds\\
  &\leq   \frac12  \iint_{Q_t} |\mathbf{b}_t|^2 dxdt+C\int_0^t\|u^2\|_{L^\infty(\Omega)} \int_{\Omega}  |\mathbf{b}_x|^2 dxds+C\left(\iint_{Q_t}u_{xx}^2 dxds\right)^{1/2}.
\end{split}\end{equation*}
Substituting these estimates into \eqref{b9}, the Gronwall inequality yields the proof of the lemma.
\end{proof}

We are now ready to prove the following estimates.

\begin{lemma}\label{2.10}
Let the assumptions in Theorem \ref{existencethm} hold. Then $\|(u,\mathbf{b})\|_{L^\infty(Q_T)}\leq C$, and
\begin{equation}\label{w1112}
\begin{split}
&\sup\limits_{0<t<T}\int_\Omega(\rho_t^2+u_x^2+\theta^2)dx +\iint_{Q_T} \big( u_t^2 +u_{xx}^2+\kappa\theta_x^2\big)dxdt\leq C,\\
& \sup\limits_{0<t<T}\int_\Omega|\mathbf{w}_x|^2\omega^2 dx+\iint_{Q_T}\big(u^2|\mathbf{w}_x|^2+ |\mathbf{w}_{t}|^2 \big)  dxdt \leq C,\\
&\sup\limits_{0<t<T}\int_\Omega|\mathbf{b}_{x}|^2dx+\iint_{Q_T}\big(|\mathbf{b}_t|^2 + |\mathbf{b}_{xx}|^2\omega^2\big) dxdt\leq C.
\end{split}
\end{equation}
 \end{lemma}
 \begin{proof}
Rewrite \eqref{e1}$_2$ as $\sqrt{\rho}u_t-\frac{\lambda}{\sqrt{\rho}}u_{xx}=-\sqrt{\rho}uu_x-
\gamma\sqrt{\rho}\theta_x-\frac{\gamma}{\sqrt{\rho}}\rho_x\theta-\frac{1}{\sqrt{\rho}}\mathbf{b}\cdot\mathbf{b}_x
$
to
 obtain
\begin{equation}\label{0u1}
\begin{split}
&  \frac{\lambda}2\int_\Omega u_x^2dx+\iint_{Q_t}\big(\rho u_t^2+\lambda^2\rho^{-1}u_{xx}^2\big)dxdt\\
&\leq \frac{\lambda}2\int_\Omega u_{0x}^2dx+C\iint_{Q_t} \big(u^2u_x^2+  \theta_x^2+\rho_x^2 \theta^2 + |\mathbf{b}\cdot\mathbf{b}_x|^2\big)dxds\\
&\leq C+C\int_0^t\|u^2\|_{L^\infty(\Omega)}\int_\Omega u_x^2dxds+C\iint_{Q_t} \theta_x^2dxds+C\int_0^t\|\theta^2\|_{L^\infty(\Omega)}\int_{\Omega} \rho_x^2  dxds\\
&\leq C+C\int_0^t\|u^2\|_{L^\infty(\Omega)}\int_\Omega u_x^2dxds+C\iint_{Q_t} \big(\theta_x^2+ \theta^2 \big) dxds,
\end{split}
\end{equation}
where we have used  \eqref{rho} and $\int_0^t\|\theta^2\|_{L^\infty(\Omega)}ds\leq C \iint_{Q_t} (\theta^2+\theta_x^2)dxds$.

Then multiplying \eqref{e1}$_5$   by $ \theta $ and
integrating over $Q_t$ give
\begin{equation}\label{theta0}
\begin{split}
    \frac12\int_\Omega \rho\theta^2dx+ \iint_{Q_t}  \kappa \theta_x^2dxds =&\frac12\int_\Omega \rho_0\theta_0^2dx -\iint_{Q_t}pu_x\theta dxds\\
   &+\iint_{Q_t}  \theta \big(\lambda u_x^2+\mu|\mathbf{w}_x|^2dx+|\mathbf{b}_x|^2\big)   dxds.
  \end{split}
\end{equation}
Note that
\begin{equation*}
\begin{split}
     -\iint_{Q_t}pu_x\theta dxds=-\gamma\iint_{Q_t}  \rho\theta^2 u_xdxds
 &\leq C \int_0^t\|u_x\|_{L^\infty(\Omega)}\int_\Omega\rho\theta^2dxds.
  \end{split}
\end{equation*}
On the other hand, by using Lemmas \ref{2.2}, \ref{2.6} and \ref{2.9} and \eqref{ux}, we obtain
\begin{equation*}
\begin{split}
     &\iint_{Q_t}  \theta \big(\lambda u_x^2+\mu|\mathbf{w}_x|^2dx+|\mathbf{b}_x|^2\big)   dxds\\
    &\leq  C\int_0^t\|u_x\|^2_{L^\infty(\Omega)}\int_\Omega\theta dx+C\int_0^t\|\theta\|_{L^\infty(\Omega)} \left
    \{\int_{\Omega}\big(\mu|\mathbf{w}_x|^2dx+|\mathbf{b}_x|^2\big)dx\right\}ds \\
    & \leq  C   +C\left(\iint_{Q_t}u_{xx}^2 dxds\right)^{1/2}.
  \end{split}
\end{equation*}
Plugging these estimates into \eqref{theta0}, the Gronwall inequality implies
\begin{equation}\label{theta33}
\begin{aligned}
     \int_\Omega  \theta^2dx+ \iint_{Q_t}  \kappa \theta_x^2dxds \leq   C+C\left(\iint_{Q_t}u_{xx}^2 dxds\right)^{1/2}.
  \end{aligned}
\end{equation}
Then \eqref{0u1} gives
\begin{equation*}\label{0u2}
\begin{split}
&  \int_\Omega u_x^2dx+\iint_{Q_t} \big(u_t^2+ u_{xx}^2\big)dxdt\\
&\leq  C +C\int_0^t \|u^2\|_{L^\infty(\Omega)}
\int_{\Omega}u_x^2  dxds +C\left(\iint_{Q_t}u_{xx}^2 dxds\right)^{1/2}\\
 &\leq  C +C\int_0^t \|u^2\|_{L^\infty(\Omega)}
 \int_{\Omega}u_x^2  dxds+\frac12\iint_{Q_t}u_{xx}^2 dxds.\\
\end{split}
\end{equation*}
Thus, the Gronwall inequality  yields
\begin{equation*}
\begin{split}
  \int_\Omega u_x^2dx+\iint_{Q_t} \big(u_t^2 + u_{xx}^2\big)dxdt   \leq  C.
\end{split}
\end{equation*}
Consequently,  \eqref{w1112} follows  from \eqref{theta33} and Lemmas 2.7-2.9.
  And the  proof of the lemma is completed.
\end{proof}

 \begin{lemma}\label{2.11}
  Under the assumptions in Theorem \ref{existencethm}, we have
\begin{equation*}\begin{split}
 \sup\limits_{0<t<T}\int_\Omega |\mathbf{w}_x | dx \leq C.
\end{split}\end{equation*}
In particular, $\|\mathbf{w}\|_{L^\infty(Q_T)}\leq C$.
\end{lemma}
 \begin{proof}
  Denote $\mathbf{z}=\mathbf{w}_x$. Differentiating \eqref{w12}  in $x$ gives
\begin{equation}\label{zx}
\begin{split}
\mathbf{z}_t=\left(\frac{\mu}{\rho}\mathbf{z}_x\right)_x-(u\mathbf{z})_x+\left(\frac{\mathbf{b}_x}{\rho}\right)_x.
\end{split}\end{equation}
Denote $\Phi_\epsilon(\cdot):\mathbb{R}^2\rightarrow \mathbb{R}^+ $ for $\epsilon \in (0, 1)$  by
$
\Phi_\epsilon(\xi)=\sqrt{\epsilon^2+|\xi|^2}, \forall\xi\in\mathbb{R}^2.
$
Observe that $\Phi_\epsilon$ has the following properties
 \begin{equation}\label{phii}
\left\{\begin{split}
  &|\xi|\leq |\Phi_\epsilon(\xi)|\leq |\xi|+\epsilon,\quad|\nabla_\xi\Phi_\epsilon(\xi)|\leq 1,\quad\lim\limits_{\epsilon\rightarrow 0^+}\Phi_\epsilon(\xi)=|\xi|,\quad\forall\xi\in\mathbb{R}^2,\\
   &   0\leq \xi\cdot\nabla_\xi\Phi_\epsilon(\xi)\leq \Phi_\epsilon(\xi),\quad \eta D_\xi^2\Phi_\epsilon(\xi)\eta^{\top}\geq 0,\quad \forall\xi, \eta\in\mathbb{R}^2,\\
\end{split}
\right.
\end{equation}
where $\xi^{\top}$ stands for the transpose of the vector
$\xi=(\xi_1,\xi_2) \in \mathbb{R}^2$, and $D_\xi^2 g$ is the
Hessian matrix of the function $g :\mathbb{R}^2\rightarrow\mathbb{R}$.

Taking the inner product of \eqref{zx} with $\nabla_\xi\Phi_\epsilon(\mathbf{z})$ and integrating over $Q_t$, we have
\begin{equation}\label{wx1}
\begin{split}
&\int_\Omega \Phi_\epsilon(\mathbf{z})dx-\int_\Omega \Phi_\epsilon(\mathbf{w}_{0x})dx\\
&=-\mu\iint_{Q_t}\frac{1}{\rho} \mathbf{z}_x D_\xi^2\Phi_\epsilon(\mathbf{z})(\mathbf{z}_x)^{\top} dxds
-\iint_{Q_t}(u\mathbf{z})_x\cdot\nabla_\xi\Phi_\epsilon(\mathbf{z})  dxds\\
&\quad+\iint_{Q_t}\Big(\frac{\mathbf{b}_x}{\rho}\Big)_x\cdot\nabla_\xi\Phi_\epsilon(\mathbf{z})  dxds
+\mu\int_0^t\frac{\mathbf{z}_x\cdot\nabla_\xi\Phi_\epsilon(\mathbf{z}) }{\rho}\Big|_{x=0}^{x=1}ds =:\sum_{j=1}^4E_j.
\end{split}
\end{equation}
From \eqref{phii},  it follows that
$$
E_1\leq 0,
$$
and
\begin{equation*}
\begin{split}
E_2=& -\iint_{Q_t}\big(u\mathbf{z}_x+u_x \mathbf{z}\big)\cdot\nabla_\xi\Phi_\epsilon(\mathbf{z})  dxds\\
=& \iint_{Q_t}\big(u_x\Phi_\epsilon(\mathbf{z})-u_x \mathbf{z}\cdot\nabla_\xi\Phi_\epsilon(\mathbf{z}) \big) dxds
\leq  C\int_0^t\|u_x\|_{L^\infty(\Omega)} \int_\Omega\Phi_\epsilon(\mathbf{z})dxds.
\end{split}
\end{equation*}
For $E_3$, using the equation \eqref{e1}$_4$ yields
\begin{equation*}
\begin{split}
 E_3 =&\iint_{Q_t}\frac{\mathbf{b}_{xx}\cdot\nabla_\xi\Phi_\epsilon(\mathbf{z})}{\rho}dxds
 -\iint_{Q_t}\frac{\mathbf{b}_{x}\cdot\nabla_\xi\Phi_\epsilon(\mathbf{z})}{\rho^2}\rho_xdxds\\[1mm]
 =&\frac{1}{\nu}\iint_{Q_t} \frac{\big[\mathbf{b}_{t}+(u\mathbf{b})_x-\mathbf{z}\big]\cdot\nabla_\xi\Phi_\epsilon(\mathbf{z})}{\rho} dxds -\iint_{Q_t}\frac{\mathbf{b}_{x}\cdot\nabla_\xi\Phi_\epsilon(\mathbf{z}) }{\rho^2}\rho_xdxds\\[1mm]
   \leq&C \iint_{Q_t}  \big[|\mathbf{b}_t|+|(u\mathbf{b})_x|+|\rho_x||\mathbf{b}_x|\big]dxds
 \leq C,
\end{split}
\end{equation*}
where we have used \eqref{phii}  and  Lemmas \ref{2.4} and \ref{2.10}.

It remains to estimate $E_4$. From \eqref{w12}, we have
\begin{equation}\label{wt}
\begin{split}
  \Big|\frac{\mu}{\rho(a,t)}\mathbf{z}_x(a,t)\Big|
  =\Big|\mathbf{w}_t(a,t)-\frac{\mathbf{b}_x(a,t)}{\rho(a,t)}\Big|
  \leq C+C |\mathbf{b}_x(a,t)|,\quad \hbox{\rm where}~a=0~\hbox{\rm
  or}~ 1.
\end{split}
\end{equation}
On the other hand, we can first integrate \eqref{e1}$_4$ from $a$ to $y
\in [0, 1]$ in $x$  and then integrate the resulting equation over
$(0, 1)$ in $y$ to obtain
 \begin{equation*}\begin{split}
\mathbf{b}_x(a,t)=&-\frac{1}{\nu}\left\{\int_0^1\hspace{-2mm}\int_a^y
\mathbf{b}_t(x,t)
dxdy+\int_0^1(u\mathbf{b}-\mathbf{w})(y,t)dy+\mathbf{w}(a,t)\right\}.
\end{split}\end{equation*}
Thus, from Lemmas \ref{2.1} and \ref{2.10}, we have that $
\int_0^T|\mathbf{b}_x(a,t)|^2dt  \leq C.
$
Then, one can derive from \eqref{wt} that $
  \int_0^T\Big|\frac{\mu}{\rho(a,t)}\mathbf{z}_x(a,t)\Big|dt \leq C+C\int_0^T|\mathbf{b}_x(a,t)| dt\leq C.
$
Hence
\begin{equation*}
\begin{split}
  E_4\leq C\int_0^T\left\{\left|\frac{\mu}{\rho(1,t)}\mathbf{z}_x(1,t)\right|
  +\left|\frac{\mu}{\rho(0,t)}\mathbf{z}_x(0,t)\right|\right\}dt  \leq C.
\end{split}
\end{equation*}
Substituting the above estimates in \eqref{wx1} and using the
Gronwall inequality, we obtain
$$\int_\Omega \Phi_\epsilon(\mathbf{z})dx\leq C+\int_\Omega \Phi_\epsilon(\mathbf{w}_{0x})dx.
$$
Taking the limit  $\epsilon\rightarrow 0$ yields  $\int_\Omega |\mathbf{w}_x|dx\leq C.
$
This and $\int_\Omega |\mathbf{w}|^2dx\leq C$  imply that
$|\mathbf{w}|\leq C$, and it completes
the proof of the lemma.
 \end{proof}

\begin{lemma}\label{2.12}Under the assumptions in Theorem \ref{existencethm}, we have
\begin{equation*}\begin{split}
 \iint_{Q_T}\big(|\mathbf{b}_x |^4+|\mathbf{b}_x||\mathbf{b}_{xx}|\big)dxdt+\int_0^T\|\mathbf{b}_x\|_{L^\infty(\Omega)}^2dt \leq C.
\end{split}\end{equation*}
\end{lemma}
 \begin{proof}
 For any fixed $ z \in [0, 1]$, we first integrate \eqref{e1}$_4$ from $z$ to $y \in [0, 1]$ in $x$, and then integrate the resulting equation over $(0, 1)$ in $y$ to have
 \begin{equation*}\label{equ4}
 \begin{split}
\mathbf{b}_x(z,t)=&-\frac{1}{\nu}\left\{\int_0^1\hspace{-2mm}\int_z^y
\mathbf{b}_t(x,t)
dxdy+\int_0^1(u\mathbf{b}-\mathbf{w})(y,t)dy-(u\mathbf{b}-\mathbf{w})(z,t)\right\},
\end{split}\end{equation*}
  which together with Lemmas \ref{2.10} and \ref{2.11}  imply that $\int_0^T\|\mathbf{b}_x\|_{L^\infty(\Omega)}^2dt \leq C.$

Combining Lemmas \ref{2.10} and \ref{2.11} gives
 \begin{equation*}\label{bxxbound}
 \begin{split}
  \iint_{Q_T}|\mathbf{b}_x||\mathbf{b}_{xx}|dxdt
 &=\frac{1}{\nu}\iint_{Q_T}|\mathbf{b}_x||\mathbf{b}_{t}+(u\mathbf{b})_x-\mathbf{w}_x|dxdt\\
  &\leq C+C\int_0^T\|\mathbf{b}_x\|_{L^\infty(\Omega)}\int_\Omega|\mathbf{w}_x|dxdt \leq
 C,
\end{split}\end{equation*}
and
\begin{equation*}\label{bxbound}
\begin{split}
\iint_{Q_T}|\mathbf{b}_x |^4dxdt\leq C\int_0^T\|\mathbf{b}_x\|_{L^\infty(\Omega)}^2\int_\Omega|\mathbf{b}_x|^2dxdt\leq   C.
\end{split}\end{equation*}
The proof of the lemma  is then completed.
\end{proof}

Now some estimates given in Lemmas \ref{2.6} and \ref{2.10} can be improved
as follows.
 \begin{lemma}\label{2.13}Under the assumptions in Theorem \ref{existencethm}, we have
\begin{equation*}\label{we1}
\begin{split}
&
\sqrt{\mu}\sup\limits_{0<t<T}\int_\Omega|\mathbf{w}_x|^2dx+\mu^{3/2}
\iint_{Q_T}
 |\mathbf{w}_{xx}|^2 dxdt \leq  C,\\
  & \sup\limits_{0<t<T}\int_\Omega|\mathbf{w}_x|^2\omega  dx+\iint_{Q_T} \big(\mu |\mathbf{w}_{xx}|^2+|\mathbf{b}_{xx}|^2\big)
  \omega dxdt\leq C.
\end{split}
\end{equation*}
 \end{lemma}
\begin{proof}
For the first estimate, we can use an argument similar to  Lemma 2.6.
The key point is to estimate the  term
$-\mu\iint_{Q_t}\frac{1}{\rho}\mathbf{b}_x\cdot\mathbf{w}_{xx}dxds $
in \eqref{0v100}. Indeed, we
have
\begin{equation}\label{w0}
\begin{split}
& -\mu\iint_{Q_t}\frac{1}{\rho}\mathbf{b}_x\cdot\mathbf{w}_{xx}dxds\\[1mm]
 &=\mu\iint_{Q_t}\frac{\mathbf{b}_{xx}\cdot\mathbf{w}_{x}}{\rho}dxds
 -\mu\iint_{Q_t}\frac{\mathbf{b}_{x}\cdot\mathbf{w}_{x}}{\rho^2}\rho_xdxds
 -\mu\int_0^t \frac{\mathbf{b}_x\cdot\mathbf{w}_{x}}{\rho}\bigg|_{x=0}^{x=1} ds\\
 &\leq C\mu\iint_{Q_t}| \mathbf{b}_{xx}|^2 dxds+C\mu\iint_{Q_t}|\mathbf{w}_{x}|^2dxds+\mu\iint_{Q_t}|\mathbf{b}_{x}|^2 \rho_x^2dxds\\[1mm]
 &\quad+C\mu\left(\int_0^t\|\mathbf{b}_{x}\|_{L^\infty(\Omega)}^2 ds\right)^{1/2}\left(\int_0^t\|\mathbf{w}_{x}\|_{L^\infty(\Omega)}^2 ds\right)^{1/2}\\
 &\leq C\mu + C\mu\iint_{Q_t}| \mathbf{b}_{xx}|^2 dxds+C\mu\iint_{Q_t}|\mathbf{w}_{x}|^2dxds
 +C\mu \left(\int_0^t\|\mathbf{w}_{x}\|_{L^\infty(\Omega)}^2 ds\right)^{1/2},\\
  \end{split}
\end{equation}
where we have used the fact that  by Lemmas \ref{2.4} and \ref{2.12}, it holds
\begin{equation*}
\begin{split}
  \iint_{Q_t}|\mathbf{b}_{x}|^2|\rho_x|^2dxds \leq
  \int_0^t\|\mathbf{b}_{x}\|_{L^\infty(\Omega)}^2\int_\Omega\rho_x^2dxds\leq
  C.
  \end{split}
\end{equation*}
By \eqref{wx2} and Lemma 2.6, we have
\begin{equation}\label{w01}
\begin{split}
  &\mu \left(\int_0^t\|\mathbf{w}_{x}\|_{L^\infty(\Omega)}^2 ds\right)^{1/2}\\
   &\leq C\sqrt{\mu}+C \mu^{1/4}\left(\mu\iint_{Q_t}|\mathbf{w}_{x}|^2dxds\right)^{1/4}
   \left(\mu^2\iint_{Q_t}|\mathbf{w}_{xx}|^2dxds\right)^{1/4}\\[1mm]
&\leq  \frac{C\sqrt{\mu}}{\epsilon}+\epsilon \mu\iint_{Q_t}
|\mathbf{w}_{x}|^2dxds+\epsilon \mu^2
\iint_{Q_t}\frac{1}{\rho}|\mathbf{w}_{xx}|^2dxds,~ \forall
\epsilon \in (0,1).
\end{split}
\end{equation}
From Lemma \ref{2.10}, we have
\begin{equation}\label{bxx}
\begin{split}
 & \iint_{Q_t}|\mathbf{b}_{xx}|^2dxds \leq C+ C\iint_{Q_t}  |\mathbf{w}_x|^2dxds.
\end{split}
\end{equation}
  Thus, \eqref{bxx} follows from the Gronwall inequality.

Inserting  the above estimates into \eqref{w0} and taking a small
$\epsilon>0$, we have
\begin{equation*}
\begin{split}
  -\mu\iint_{Q_t}\frac{1}{\rho}\mathbf{b}_x\cdot\mathbf{w}_{xx}dxdt \leq C\sqrt{\mu} +C\mu\iint_{Q_t}|\mathbf{w}_{x}|^2dxdt+\frac{\mu^2}{4}
\iint_{Q_t}\frac{1}{\rho}|\mathbf{w}_{xx}|^2dxds.
  \end{split}
\end{equation*}
Then, an argument similar to Lemma \ref{2.6} leads to
\begin{equation*}
\begin{split}
  &\mu\int_\Omega|\mathbf{w}_x|^2dx
 +\mu^2\iint_{Q_t} |\mathbf{w}_{xx}|^2dxds
  \leq C\sqrt{\mu}
 +C\int_0^t\big(1+\|u_x\|_{L^\infty(\Omega)}\big)\left(\mu\int_{\Omega}|\mathbf{w}_x|^2dx\right)ds.
  \end{split}
\end{equation*}
Thus,   the first estimate of this lemma follows from the Gronwall inequality and
\eqref{u0}.

The second estimate can proved by the arguments similar to Lemma \ref{2.7} and
\eqref{w1}$_1$ and by the first estimate and Lemmas
\ref{2.10}-\ref{2.12}. In fact, this can be done  by using $\omega$ instead of
$\omega^2$ in \eqref{b111} and   \eqref{0v10} and noticing the following facts:
\begin{equation*} \begin{split}
  & \mu\iint_{Q_T}|\mathbf{w}_x\cdot\mathbf{w}_{xx}|dxdt \leq C\sqrt{\mu}\iint_{Q_T}|\mathbf{w}_x|^2dxdt+C\mu^{3/2}\iint_{Q_T}|\mathbf{w}_{xx}|^2dxdt\leq C,\\[1mm]
  & \iint_{Q_T}|\mathbf{b}_x\cdot\mathbf{w}_{x}|dxdt  \leq C\int_0^T\|\mathbf{b}_x\|_{L^\infty(\Omega)}\int_\Omega|\mathbf{w}_x|dxdt \leq
 C.
\end{split}\end{equation*}
 The proof is completed.
\end{proof}

As a consequence of Lemma \ref{2.13} and \eqref{w01}, we   have
\begin{equation}\label{wx4}
\begin{split}
   \mu^{3/2}\iint_{Q_T} |\mathbf{w}_{x}|^4dxdt  \leq C \mu \int_0^T \|\mathbf{w}_{x}\|_{L^\infty(\Omega)}^2 \left(\sqrt{ \mu}\int_{Q_T}|\mathbf{w}_x|^2dx\right)dt\leq C.
  \end{split}
\end{equation}

With the above estimates, we are now ready to show
 the upper bound estimate on   $\theta$.

\begin{lemma}\label{2.14}
 Let the assumptions in Theorem \ref{existencethm}  hold. Then  $\theta \leq C.$
\end{lemma}
\begin{proof}
Rewrite the equation \eqref{e1}$_4$  into
\begin{equation}\label{theta9}
 \begin{split}
 \theta_t= a(x,t)\theta_{xx}+b(x,t)\theta_x+c(x,t) \theta+f(x,t),
\end{split}
\end{equation}
where $
a =\rho^{-1}\kappa, b =\rho^{-1}\kappa_x-u, c =-\gamma
u_x, f =\rho^{-1} (\lambda
u_x^2+\mu|\mathbf{w}_x|^2+\nu|\mathbf{b}_x|^2).
$

Set $z=\theta_x$. Differentiating the equation \eqref{theta9} in $x$ yields
\begin{equation}\label{z1}
 \begin{split}
 z_t=(a z_x)_x+(b z)_x+(c \theta)_x +f_x.
\end{split}
\end{equation}
For $\epsilon \in (0, 1)$, denote $\varphi_\epsilon:\mathbb{R}\rightarrow \mathbb{R}^+$   by
$
\varphi_\epsilon(s)=\sqrt{s^2+\epsilon^2},
$
satisfying
\begin{equation*}\label{phi}
  \begin{split}
 &\varphi_\epsilon'(0)=0,\quad|\varphi_\epsilon'(s)|\leq 1,  \quad0\leq s^2\varphi_\epsilon''(s)\leq \epsilon, \quad\lim\limits_{\epsilon\rightarrow 0} \varphi_\epsilon(s)=|s|.
\end{split}
\end{equation*}
Multiplying  \eqref{z1} by $\varphi_\epsilon'(z)$, integrating over
$Q_t$, and noticing
$\varphi_\epsilon'(z)|_{x=0,1}=\varphi_\epsilon'(\theta_x)|_{x=0,1}=0$, $\varphi_\epsilon''(s)\geq 0$ and $|\varphi_\epsilon'(s)|\leq 1$,
we have
\begin{equation}\label{z2}
 \begin{split}
 &\int_\Omega \varphi_\epsilon(z)dx-\int_\Omega \varphi_\epsilon(\theta_{0x})dxds\\
 &=-\iint_{Q_t} \big(a z_x^2 +b z z_x\big)\varphi_\epsilon''(z)dxds +\iint_{Q_t}  [(c\theta)_x +f_x]\varphi_\epsilon'(z)dxds,\\
 &=-\iint_{Q_t} a\left[\Big(z_x +\frac{bz}{2a}\Big)^2\varphi_\epsilon''(z) +  \frac{b^2}{4a} z^2\varphi_\epsilon''(z)\right]dxds +\iint_{Q_t}  [(c\theta)_x +f_x]\varphi_\epsilon'(z)dxds,\\
 &\leq\iint_{Q_t} \frac{b^2}{4a} z^2\varphi_\epsilon''(z)dxds +\iint_{Q_t}  (|(c\theta)_x| +|f_x|) dxds.\\
 \end{split}
\end{equation}
Thus, letting $\epsilon\rightarrow 0$ in \eqref{z2} and using $0\leq   s^2\varphi_\epsilon''(s)\leq \epsilon$ and $\lim\limits_{\epsilon\rightarrow 0} \varphi_\epsilon(s)=|s|$, we have
\begin{equation}\label{theta10}
 \begin{split}
  \int_\Omega |\theta_x|dx\leq \int_\Omega|\theta_{0x}|dx  + \iint_{Q_t}(|c \theta_x|+|c_x \theta|+|f_x|)dxds.
\end{split}
\end{equation}
By Lemma \ref{2.10}, we have
\begin{equation*}
 \begin{split}
 &\iint_{Q_T}|c_x\theta| dxdt\leq C\left(\iint_{Q_T}u_{xx}^2dxdt\right)^{1/2}\left(\iint_{Q_T}\theta^2dxdt\right)^{1/2}\leq C,\\
&\iint_{Q_T}|c \theta_x| dxdt\leq C\left(\iint_{Q_T}u_{x}^2dxdt\right)^{1/2}\left(\iint_{Q_T}\theta_x^2dxdt\right)^{1/2}\leq C.
\end{split}
\end{equation*}
By   Lemmas \ref{2.4}, \ref{2.10} and \ref{2.12}, and \eqref{wx4}, we
obtain
\begin{equation*}\label{theta6}
 \begin{split}
 \iint_{Q_T}|f_x| dxdt
&\leq C\iint_{Q_T}\big(|u_x||u_{xx}|+\mu|\mathbf{w}_x\cdot\mathbf{w}_{xx}|+ |\mathbf{b}_x\cdot\mathbf{b}_{xx}|\big)dxdt\\
&\quad+C\iint_{Q_T}\big(u_x^2+\mu|\mathbf{w}_x|^2+ |\mathbf{b}_x|^2\big)|\rho_x|dxdt \leq C.
\end{split}
\end{equation*}
 Substituting the above estimates into \eqref{theta10} yields that
 $\int_\Omega|\theta_x|dx\leq C,$
which  yields the upper bound of $\theta$ because $\int_\Omega\theta dx\leq C$.
\end{proof}

The following estimtes are about the derivatives of $\theta$.

\begin{lemma}\label{2.15}
Under the assumptions in Theorem \ref{existencethm}, we have
\begin{equation*}\begin{split}
\sup\limits_{0<t<T}\int_\Omega \theta_x^2dx+\iint_{Q_T} \big(\theta_t^2+\theta_{xx}^2\big)dxdt\leq C.
\end{split}\end{equation*}
\end{lemma}
\begin{proof} Rewrite the equation \eqref{e1}$_5$ as
\begin{equation}\label{theta1}
\begin{split}
 \rho\theta_t-(\kappa\theta_x)_x=\lambda u_x^2+\mu
|\mathbf{w}_x|^2+\nu |\mathbf{b}_x|^2-\rho
 u\theta_x-\gamma\rho\theta u_x=:f.
\end{split}\end{equation}
We first estimate  $\|f\|_{L^2(Q_T)}$. By
\eqref{bxbound}, \eqref{wx4}  and Lemmas \ref{2.10}, \ref{2.14} and \ref{2.12},
we have
\begin{equation}\label{f}
\begin{split}
 \iint_{Q_T}f^2dxdt\leq &
   C\iint_{Q_T}(u_x^4+\mu^2|\mathbf{w}_x|^4+\nu^2|\mathbf{b}_x|^4+\rho^2u^2\theta_x^2+\rho^2u_x^2\theta^2)dxdt \leq C.
\end{split}\end{equation}
 Multiplying  \eqref{theta1} by $\kappa\theta_t$ and  integrating over $Q_t$, we have
\begin{equation}\label{theta2}
\begin{split}
 &\iint_{Q_t} \rho\kappa\theta_t^2 dxdt+\iint_{Q_t}  \kappa\theta_x (\kappa\theta_t)_x dxdt
 =\iint_{Q_t}f\kappa\theta_tdxdt.
\end{split}\end{equation}
Observe that
$ (\kappa\theta_t)_x=(\kappa\theta_x)_t+\kappa_\rho\rho_x\theta_t+\kappa_\rho\theta_x(\rho_xu+\rho u_x),$
then
\begin{equation*}\label{theta3}
\begin{split}
  \iint_{Q_t}  \kappa\theta_x (\kappa\theta_t)_x dxdt=& \frac12\int_\Omega \kappa^2\theta_x^2dx-\frac12\int_\Omega \kappa^2(\rho_0,\theta_0)\theta_{0x}^2dx\\
&+\iint_{Q_t}\Big[\kappa\kappa_\rho\rho_x\theta_x\theta_t+\kappa\kappa_\rho\theta_x^2(\rho_xu+\rho
u_x)\Big]dxdt.
\end{split}\end{equation*}
Then substituting this into \eqref{theta2}  yields
\begin{equation}\label{theta8}
\begin{split}
 &\iint_{Q_t} \rho\kappa\theta_t^2 dxdt+\int_\Omega
 \kappa^2\theta_x^2dx\\
& \leq C-2\iint_{Q_t}\Big[\kappa\kappa_\rho\rho_x\theta_x\theta_t
 +\kappa\kappa_\rho\theta_x^2(\rho_xu+\rho u_x)-f\kappa\theta_t\Big]dxdt.\\
\end{split}\end{equation}
By  $C^{-1}\leq \rho, \theta \leq C$ and \eqref{kappa},
we have that $ \kappa_1\leq \kappa \leq C$ and $|\kappa_\rho| \leq C. $ By
Young inequality, \eqref{rho}, \eqref{f} and Lemma \ref{2.10}, we
obtain
\begin{equation}\label{theta4}
\begin{split}
 &-2\iint_{Q_t}\Big[\kappa\kappa_\rho\rho_x\theta_x\theta_t+\kappa\kappa_\rho\theta_x^2(\rho_xu+\rho u_x)
 -f\kappa\theta_t\Big]dxdt\\
 & \leq C+\frac14 \iint_{Q_t} \rho\kappa\theta_t^2 dxdt+C\iint_{Q_t} (\kappa\theta_x)^2(\rho_x^2+|\rho_x|+|u_x|) dxds\\
 & \leq C+\frac14 \iint_{Q_t} \rho\kappa\theta_t^2 dxdt+C\int_0^t\|\kappa\theta_x\|_{L^\infty(\Omega)}^2 ds.
\end{split}\end{equation}
Now we  estimate the second integral on the right hand side
of \eqref{theta4}. By the embedding theorem  and  Young inequality, we have
\begin{equation*}\begin{split}
 \int_0^t\|\kappa\theta_x\|_{L^\infty(\Omega)}^2 ds \leq& \iint_{Q_t} |\kappa\theta_x|^2 dxds+2\iint_{Q_t} | \kappa\theta_x ||(\kappa\theta_x)_x| dxds\\
 \leq & \frac{C}{\epsilon}+\frac{\epsilon}{2}\iint_{Q_t} \big|(\kappa\theta_x)_x\big|^2 dxds, \quad \forall \epsilon>0.
\end{split}\end{equation*}
 Then,  from \eqref{theta1}, we obtain
\begin{equation*}\begin{split}
 \int_0^t\|\kappa\theta_x\|_{L^\infty(\Omega)}^2 ds \leq  \frac{C}{\epsilon}
 + \epsilon\iint_{Q_t} (\rho^2\theta_t^2+f^2)dxdt.
\end{split}\end{equation*}
Plugging it into \eqref{theta4}, taking a small $\epsilon>0$ and
using \eqref{f}, we obtain
\begin{equation*}
\begin{split}
  -2\iint_{Q_t}\Big[\kappa\kappa_\rho\rho_x\theta_x\theta_t+\kappa\kappa_\rho\theta_x^2(\rho_xu+\rho u_x)
  -f\kappa\theta_t\Big]dxdt
   \leq C+\frac12 \iint_{Q_t} \rho\kappa\theta_t^2 dxdt.
\end{split}\end{equation*}
This together with \eqref{kappa} and \eqref{theta8} give
\begin{equation}\label{theta}
\begin{split}
\sup\limits_{0<t<T}\int_\Omega \theta_x^2dx+\iint_{Q_T}  \theta_t^2
dxdt\leq C.
\end{split}\end{equation}

By \eqref{theta}  and Lemma \ref{2.14}, 
it follows from  \eqref{e1}$_5$  that $\|\theta_{xx}\|_{L^2(Q_T)} \leq C$,
and this completes the proof of the lemma.
\end{proof}

 In summary, all the estimates in \eqref{ve} have been proved.

\subsection{Proof of Theorem 1.1(ii)}
 Similar to \eqref{rho}$_2$, one can show that
\begin{equation}\label{base5}
\begin{aligned}
 \|(u, \mathbf{b}, \theta)\|_{C^{1/2, 1/4}(\overline Q_T)}\leq C,\quad
 \|\mathbf{w}\|_{C^{1/2, 1/4}([\delta, 1-\delta]\times[0, T])}\leq C,~\forall\delta\in\big(0, 1/2\big) .
  \end{aligned}
  \end{equation}
From  \eqref{rho}$_2$,   \eqref{bxbound},  \eqref{base5} and Lemmas \ref{2.2}, \ref{2.4} and
\ref{2.10}-\ref{2.15}, it follows that there exist a subsequence
$\mu_j \rightarrow 0$
such that  the corresponding  solution
to the problem \eqref{e1}-\eqref{e4} with $\mu=\mu_j $, still denoted by $(\rho,u,\mathbf{w},\mathbf{b}, \theta)$, converges to
$(\overline\rho, \overline u, \overline{\mathbf{w}},
\overline{\mathbf{b}}, \overline\theta) \in \mathbb{F}$ in the following sense:
\begin{equation*}\label{rate}
 \begin{split}
 &(\rho,u,\mathbf{b}, \theta)\rightarrow (\overline\rho,\overline u,\overline{\mathbf{b}}, \overline\theta)~~
 \hbox{\rm in}~~C^\alpha(\overline Q_T),~\forall\alpha\in(0, 1/4),\\[1mm]
 &(\rho_t,\rho_x,u_x,\mathbf{b}_x,\theta_x)\rightharpoonup(\overline\rho_t, \overline\rho_x,\overline u_x,
 \overline{\mathbf{b}}_x,\overline\theta_x)~~ \hbox{\rm weakly}-*~\hbox{\rm in}~ L^\infty(0, T; L^2(\Omega)),\\[1mm]
  &(u_t,\mathbf{b}_t,\theta_t,u_{xx}, \theta_{xx})\rightharpoonup(\overline u_t,\overline{\mathbf{b}}_t,
  \overline\theta_t,\overline u_{xx},  \overline\theta_{xx})~~ \hbox{\rm  weakly in}~~L^2(Q_T),\\[1mm]
  & \mathbf{b}_{xx}  \rightharpoonup  \overline{\mathbf{b}}_{xx}~~ \hbox{\rm weakly in}~~L^2((\delta,1-\delta)\times(0, T)),~\forall \delta\in(0, 1/2),\\
      \end{split}
  \end{equation*}
 and
  \begin{equation*}
   \begin{split}
 & \mathbf{w}  \rightarrow   \overline{\mathbf{w}}\quad\hbox{\rm  in}~~C^\alpha([\delta,1-\delta]\times[0, T]),~\forall \delta\in\big(0, 1/2\big),~\alpha\in(0, 1/4),\\
 & \mathbf{w}_t  \rightharpoonup   \overline{\mathbf{w}}_t   \quad\hbox{\rm weakly in}~~L^2(Q_T),\\
   &   \mathbf{w}_x \rightharpoonup  \overline{\mathbf{w}}_x\quad\hbox{\rm weakly}-*~ \hbox{\rm in}~ L^\infty(0, T; L^2(\delta,1-\delta)),~\forall \delta\in(0, 1/2),\\
 &\mathbf{w}\rightarrow  \overline{\mathbf{w}}~~ \hbox{\rm strongly in}~~L^r(Q_T),\quad\forall r \in [1, +\infty),\\
 &\sqrt{\mu}\|\mathbf{w}_x\|_{L^2(Q_T)} \rightarrow 0.
   \end{split}
    \end{equation*}

    We now show the strong convergence of $(u_x,\mathbf{b}_x,\theta_x)$ in $L^2(Q_T)$. Multiplying \eqref{e1}$_2$ with $\mu=\mu_j$ by $(u-\overline u)$ and integrating over $Q_T$, we have
\begin{equation*}
\begin{split}
  &\lambda\iint_{Q_T} \big(u_{x}-\overline u_x\big)^2dxdt+\lambda\iint_{Q_T}\overline u_x\big(u_{x}-\overline u_x\big)dxdt\\
  &=-\iint_{Q_T}\Big[(\rho u)_{t}
  +\big(\rho u^2+\gamma\rho\theta +\frac12|\mathbf{b}|^2\big)_x\Big](u-\overline u)   dxdt.
 \end{split}
\end{equation*}
Then, from Lemmas \ref{2.4}, \ref{2.10}  and  \ref{2.14}, we obtain
$$
u_{x}\rightarrow\overline u_x~\hbox{\rm strongly in}~ L^2(Q_T) ~\hbox{\rm as}~\mu_j\rightarrow 0.
$$
Similarly, one has
\begin{equation*}
\begin{split}
&(\mathbf{b}_{x},\theta_{x})\rightarrow (\overline{\mathbf{b}}_x,\overline \theta_x)~\hbox{\rm strongly  in}~ L^2(Q_T) ~\hbox{\rm as}~\mu_j\rightarrow 0.
 \end{split}
\end{equation*}

Thus, one can check that  $(\overline\rho,\overline u,
\overline{\mathbf{w}}, \overline{\mathbf{b}}, \overline\theta)$
  is a solution to the problem \eqref{e1}-\eqref{e4}
with   $\mu=0$ in the  sense of distribution.    On the other hand, one can see from Theorem 1.1(iii) that
problem \eqref{e1}-\eqref{e4}
with  $\mu=0$ admits at most one solution in  $\mathbb{F}$. Therefore, the above convergence holds for any $\mu_j\rightarrow 0$. The proof of Theorem 1.1(ii) is then completed.

\subsection{Proof of Theorem 1.1(iii)}
The proof is divided into several steps.  For convenience,  set
\begin{equation*}
\begin{split}
&\widetilde{\rho}=\rho-\overline\rho,\quad \widetilde{u}= u-\overline
u,\quad\widetilde{\mathbf{w}}=\mathbf{w}-\overline{\mathbf{w}},
\quad\widetilde{\mathbf{b}}=\mathbf{b}-\overline{\mathbf{b}},\quad
\widetilde{\theta}=\theta-\overline\theta,\\
&\mathbb{H}(t)=\|(\widetilde{\rho}, \widetilde{u},\widetilde{\mathbf{w}}, \widetilde{\mathbf{b}}, \widetilde{\theta})\|_{L^2(\Omega)}^2, D(t)=1+ \|(u_x,\mathbf{b}_x, \overline u_x, \overline{\mathbf{b}}_x,\overline{\theta}_x)\|_{L^\infty(\Omega)}^2+\|(\overline u_t, \overline\theta_t, \overline u_x, \overline{\mathbf{b}}_x,\overline{\theta}_x)\|_{L^2(\Omega)}^2.
 \end{split}
\end{equation*}
Clearly, $D(t)\in L^1(0, T)$.\\

\indent{\bf Step 1} Claim that
\begin{equation}\label{rho4}
\begin{split}
  \int_\Omega \widetilde{\rho}^2dx\leq  \epsilon\iint_{Q_t} \widetilde{u}_x^2dxds+\frac{C}{\epsilon}\int_0^tD(s)\mathbb{H}(s)ds,~~ \forall \epsilon>0.
  \end{split}
\end{equation}
From the equations of $\rho$ and $\overline\rho$, we have that $
\widetilde{\rho}_t=-\big(\rho \widetilde{u} +\overline
u\widetilde{\rho}\big)_x. $
Multiplying it by $\widetilde{\rho}$ and integrating over $Q_t$, and using
the Young inequality, we obtain
\begin{equation*}\label{rho1}
\begin{split}
  \frac12\int_\Omega \widetilde{\rho}^2dx=&-\iint_{Q_t}\big(\rho \widetilde{u}_x\widetilde{\rho}+\rho_x\widetilde{u}\widetilde{\rho} \big)dxds -\frac12\iint_{Q_t}\overline
  u_x\widetilde{\rho}^2dxds\\
    \leq&\frac{\epsilon}{4}\iint_{Q_t} \widetilde{u}_x^2dxds+C\iint_{Q_t} \widetilde{u}^2\rho_x^2dxds  +\frac{C}{\epsilon}
    \int_0^t(1+\|\overline u_x\|_{L^\infty(\Omega)})\int_\Omega\widetilde{\rho}^2dxds, \forall \epsilon>0.
  \end{split}
\end{equation*}
From \eqref{rho}, we have, for the above $\epsilon$,
\begin{equation*}\label{rho3}
\begin{split}
   C\iint_{Q_t} \widetilde{u}^2\rho_x^2dxds
   \leq C\int_0^t\|\widetilde{u}\|_{L^\infty(\Omega)}^2 ds\leq \frac{\epsilon}{4}\iint_{Q_t} \widetilde{u}_x^2dxds+\frac{C}{\epsilon}\iint_{Q_t} \widetilde{u}^2dxds.
  \end{split}
\end{equation*}
Thus, the claim \eqref{rho4} holds.\\

\indent{\bf Step 2}  Claim that
\begin{equation}\label{u14}
\begin{split}
   \int_\Omega  \widetilde{u}^2dx+ \iint_{Q_t} \widetilde{u}_x ^2dxds\leq   C \int_0^tD(s)\mathbb{H}(s)ds.
\end{split}
\end{equation}
From the equations of $u$ and $\overline u$, we have
  \begin{equation*}
\begin{split}
 &\big(\rho \widetilde{u} \big)_t+\big(\rho u\widetilde{u}\big)_x+\widetilde{\rho}\overline u_t
 +(\rho u-\overline\rho~\overline u)\overline u_x +\gamma(\rho\theta-\overline\rho\overline\theta)_x
 +\frac12(|\mathbf{b}|^2 -|\overline{\mathbf{b}}|^2 )_x =\lambda \widetilde{u}_{xx}.
  \end{split}
  \end{equation*}
Multiplying it by $\widetilde{u}$ and integrating  over $Q_t$ give
\begin{equation}\label{uu}
\begin{split}
  &\frac12\int_\Omega\rho \widetilde{u}^2dx+\lambda\iint_{Q_t} \widetilde{u}_x ^2dxds\\
  &=-\iint_{Q_t}\widetilde{\rho} \overline u_t \widetilde{u}dxds
 -\iint_{Q_t}(\rho u-\overline\rho~\overline u)\overline u_x\widetilde{u} dxds
 +\gamma\iint_{Q_t} (\rho\theta-\overline\rho\overline\theta) \widetilde{u}_x dxds
 \\
 &\quad+\frac12\iint_{Q_t} (|\mathbf{b}|^2 -|\overline{\mathbf{b}}|^2 ) \widetilde{u}_x dxds   =:\sum_{i=1}^4 I_i.
\end{split}
\end{equation}
Observe that $\rho u-\overline\rho~\overline u=\rho \widetilde{u} +\overline u \widetilde{\rho} $ and
$\rho \theta-\overline\rho\overline\theta =\rho\widetilde{\theta}+\overline\theta \widetilde{\rho} $. Then
\begin{equation}\label{equ700}
\begin{split}
 I_1   & \leq  C \int_0^t\left(\int_\Omega\widetilde{\rho}^2 dx\right)^{1/2}
  \left(\int_\Omega \overline u_t^2 dx\right)^{1/2}  \|\widetilde{u}\|_{L^\infty(\Omega)} dt\\
    &\leq C \int_0^t \left(\int_\Omega \overline u_t^2 dx\right)\left(\int_\Omega\widetilde{\rho}^2 dx\right)dt
    +C\int_0^t \|\widetilde{u}\|_{L^\infty(\Omega)}^2 ds\\
    & \leq C \int_0^tD(s)\mathbb{H}(s)ds
    + \frac\lambda4\iint_{Q_t}   \widetilde{u}_x ^2 dxds,
\end{split}
\end{equation}
 \begin{equation}\label{u1}
\begin{split}
 I_2
  &\leq  C \iint_{Q_t}\Big(|\overline{u}_x|\widetilde{u}^2+|\overline{u}_x|\widetilde{u}\widetilde{\rho}\Big)dxds\\
  &\leq C\iint_{Q_t}\Big((|\overline{u}_x|^2+1)\widetilde{u}^2+ \widetilde{\rho}^2\Big)dxds  \leq C \int_0^tD(s)\mathbb{H}(s)ds,
\end{split}
\end{equation}
and
\begin{equation}\label{u2}
\begin{split}
  I_3&\leq  \frac\lambda4\iint_{Q_t}   \widetilde{u}_x ^2 dxds+C\iint_{Q_t}  (\widetilde{\theta}^2+\widetilde{\rho}^2) dxds \leq \frac\lambda4\iint_{Q_t}   \widetilde{u}_x ^2 dxds+C \int_0^tD(s)\mathbb{H}(s)ds.
\end{split}
\end{equation}
Using  $\|(\mathbf{b},\overline{\mathbf{b}})\|_{L^\infty(Q_T)}\leq C$, we obtain
\begin{equation}\label{b01}
\begin{split}
  I_4&\leq  \frac\lambda4\iint_{Q_t}   \widetilde{u}_x ^2 dxds+C\iint_{Q_t} |\widetilde{\mathbf{b}}|^2 dxds\leq \frac\lambda4\iint_{Q_t}   \widetilde{u}_x ^2 dxds+C \int_0^tD(s)\mathbb{H}(s)ds.
\end{split}
\end{equation}
Substituting \eqref{equ700}-\eqref{b01} into \eqref{uu} completes the proof of \eqref{u14}.\\

\indent{\bf Step 3} Claim that
\begin{equation}\label{Theta}
\begin{split}
 &\int_\Omega\widetilde{\theta}^2dx+ \iint_{Q_t} \widetilde{\theta}_x^2dxds\\
 &\leq C\sqrt{\mu}
 +\epsilon\iint_{Q_t}\big(\widetilde{u}_x^2+|\widetilde{\mathbf{b}}_x|^2\big)dxds
 +\frac{C}{\epsilon}\int_0^tD(s)\mathbb{H}(s)ds,~~\forall\epsilon>0.
\end{split}
\end{equation}
From the equations of $\theta$ and $\overline\theta$, we have
 \begin{equation*}\label{ll}
\begin{split}
& \big(\rho\widetilde{\theta}\big)_t+(\rho u\widetilde{\theta})_x+\widetilde{\rho}\overline\theta_t+(\rho \widetilde{u}+\overline u \widetilde{\rho})\overline\theta_x+\gamma\rho\theta \widetilde{u}_x
 +\gamma\big(\rho\widetilde{\theta}+\widetilde{\rho}\overline\theta\big)\overline u_x \\ &=\big[\kappa(\rho,\theta)\widetilde{\theta}_x\big]_x +\big[(\kappa(\rho,\theta)-\kappa(\overline\rho,\overline\theta))\overline\theta_x\big]_x +\lambda(u_x^2-\overline u_x^2) +\mu|\mathbf{w}_x|^2+\nu(|\mathbf{b}_x|^2 -|\overline{\mathbf{b}}_x|^2).
  \end{split}
  \end{equation*}
 Multiplying it by $\widetilde{\theta}$ and integrating over $Q_t$ give
  \begin{equation}\label{21}
\begin{split}
 & \frac12\int_\Omega\rho\widetilde{\theta}^2dx+ \iint_{Q_t}\kappa\widetilde{\theta}_x^2dxds \\
 &=-\iint_{Q_t}\widetilde{\rho}\widetilde{\theta}\overline\theta_tdxdt-\iint_{Q_t}(\rho \widetilde{u}+\overline u \widetilde{\rho})\widetilde{\theta}\overline\theta_xdxds
 -\gamma\iint_{Q_t} \rho\theta \widetilde{u}_x \widetilde{\theta} dxds\\
&\quad-\gamma\iint_{Q_t}\rho \overline u_x\widetilde{\theta}^2dxds -\gamma\iint_{Q_t}\overline\theta\overline u_x\widetilde{\rho}\widetilde{\theta} dxds
-\iint_{Q_t}\overline\theta_x[\kappa(\rho,\theta)-\kappa(\overline\rho,\overline\theta)]\widetilde{\theta}_xdxds \\
&\quad+\lambda\iint_{Q_t} (u_x+\overline u_x)\widetilde{u}_x \widetilde{\theta} dxds+\mu\iint_{Q_t} |\mathbf{w}_x|^2 \widetilde{\theta} dxds\\
&\quad+\nu\iint_{Q_t}(|\mathbf{b}_x|^2 -|\overline{\mathbf{b}}_x|^2)\widetilde{\theta} dxds=:\sum_{i=1}^9 E_i.
\end{split}
\end{equation}
  By the H\"{o}lder inequality and Young inequality, we have
\begin{equation*}\label{u8}
\begin{split}
 &E_1+E_2+E_5 \leq  C\int_0^t\left(\int_\Omega\big(\widetilde{\rho}^2+ \widetilde{u}^2\big)dx\right)^{1/2}\left(\int_\Omega\widetilde{\theta}^2(\overline\theta_t^2+\overline\theta_x^2+\overline u_x^2)dx\right)^{1/2}dt\\
 &\leq  C\int_0^t \left(\int_\Omega(\overline\theta_t^2+\overline\theta_x^2+\overline u_x^2)dx\right)^{1/2} \left(\int_\Omega(\widetilde{\rho}^2+ \widetilde{u}^2)dx\right)^{1/2}\|\widetilde{\theta}\|_{L^\infty(\Omega)}dt\\
 &\leq  C\int_0^t\left(\int_\Omega(\overline\theta_t^2+\overline\theta_x^2+\overline u_x^2)dx\right)\left(\int_\Omega\big(\widetilde{\rho}^2+ \widetilde{u}^2\big)dx\right)dt
 +\int_0^t\|\widetilde{\theta}\|_{L^\infty(\Omega)}^2ds\\
&\leq  C\int_0^tD(s)\mathbb{H}(s)ds +\frac{\kappa_1}{4}\iint_{Q_t}\widetilde{\theta}_x^2dxds.\\
\end{split}
\end{equation*}
Hence
\begin{equation*}\label{u111}
\begin{split}
   E_3+E_4 +E_7
 &\leq \epsilon\iint_{Q_t}  \widetilde{u}_x ^2dxds + \frac{C}{\epsilon}\int_0^t\big(1+\|u_x\|_{L^\infty(\Omega)}^2+\|\overline u_x\|_{L^\infty(\Omega)}^2\big)\int_{\Omega} \widetilde{\theta}^2dxds\\
 &\leq \epsilon\iint_{Q_t}  \widetilde{u}_x ^2dxds + \frac{C}{\epsilon}\int_0^tD(s)\mathbb{H}(s)ds,\quad \forall\epsilon \in (0, 1).
\end{split}
\end{equation*}
By   $C^{-1}\leq \rho, \overline\rho, \theta, \overline\theta \leq C$, we obtain that
$|\kappa(\rho,\theta)-\kappa(\overline\rho,\overline\theta)|\leq C(|\widetilde{\rho} |+|\widetilde{\theta}|).$
Thus,
 \begin{equation*}\label{u9}
\begin{split}
 E_6 \leq & \frac{\kappa_1}{4}\iint_{Q_t} \widetilde{\theta}_x^2dxds+C\iint_{Q_t}|\overline\theta_x|^2\big(\widetilde{\rho}^2+\widetilde{\theta}^2\big)dxds \\   \leq & \frac{\kappa_1}{4} \iint_{Q_t} \widetilde{\theta}_x^2dxds+C\int_0^t\|\overline\theta_x\|^2_{L^\infty(\Omega)}\int_{\Omega}\big(\widetilde{\rho}^2+\widetilde{\theta}^2\big)dxds\\
    \leq &  \frac{\kappa_1}{4} \iint_{Q_t} \widetilde{\theta}_x^2dxds+C\int_0^tD(s)\mathbb{H}(s)ds.
\end{split}
\end{equation*}
By \eqref{wx4}, we have
\begin{equation*}\label{u10}
\begin{split}
 E_8\leq  C\sqrt{\mu}.
\end{split}
\end{equation*}
For $E_9$, by noticing
$|\mathbf{b}_x|^2-|\overline{\mathbf{b}}_x|^2=(\mathbf{b}_x+\overline{\mathbf{b}}_x)\cdot\widetilde{\mathbf{b}}_x$, we obtain
\begin{equation*}
\begin{split}
 E_9\leq & \epsilon\iint_{Q_t}|\widetilde{\mathbf{b}}_x|^2dxds
 +\frac{C}{\epsilon}\int_0^t\big(\|\mathbf{b}_x\|_{L^\infty(\Omega)}^2+\|\overline{\mathbf{b}}_x\|_{L^\infty(\Omega)}^2\big)
 \int_{\Omega}\widetilde{\theta}^2dxds\\
 \leq & \epsilon\iint_{Q_t}|\widetilde{\mathbf{b}}_x|^2dxds
 +\frac{C}{\epsilon}\int_0^tD(s)\mathbb{H}(s)ds,~~\forall\epsilon>0.
\end{split}
\end{equation*}
Substituting these estimates into \eqref{21} completes the proof of
\eqref{Theta}.\\

\indent{\bf Step 4} Claim that
\begin{equation}\label{0v3}
\begin{split}
  & \int_\Omega |\widetilde{\mathbf{w}}|^2dx   \leq C\sqrt{\mu} +\epsilon\iint_{Q_t} |\widetilde{\mathbf{b}}_x|^2 dxds + \frac{C}{\epsilon}\int_0^tD(s)\mathbb{H}(s)ds,\quad\forall \epsilon>0.
  \end{split}
  \end{equation}
 From the equations for $\mathbf{w}$ and $\overline{\mathbf{w}}$, we have $\rho\widetilde{\mathbf{w}}_t+ \rho u \widetilde{\mathbf{w}}_x+  \rho \widetilde{u} \overline{\mathbf{w}}_x-\widetilde{\mathbf{b}}_x+\frac{\widetilde{\rho}}{\overline\rho}\overline{\mathbf{b}}_x= \mu \mathbf{w}_{xx}.
  $
  Taking the inner product of the identity with $\widetilde{\mathbf{w}}$ and integrating over $Q_t$ give
  \begin{equation}\label{0v33}
\begin{split}
   \frac12\int_\Omega \rho|\widetilde{\mathbf{w}}|^2dx
 & =\iint_{Q_t}  \big(\mu\mathbf{w}_{xx}\cdot\widetilde{\mathbf{w}} - \rho\widetilde{u}\overline{\mathbf{w}}_x\cdot\widetilde{\mathbf{w}}
  +\widetilde{\mathbf{b}}_x\cdot\widetilde{\mathbf{w}}-\frac{\widetilde{\rho}}{\overline\rho}
  \overline{\mathbf{b}}_x \cdot\widetilde{\mathbf{w}}\big)dxds\\
  &\leq C\mu^2\iint_{Q_t} |\mathbf{w}_{xx}|^2dxds+C\iint_{Q_t} \widetilde{u}^2|\overline{\mathbf{w}}_x|^2dxds
 \\
 &\quad +\frac{C}{\epsilon}\int_0^tD(s)\mathbb{H}(s)ds  +\epsilon\iint_{Q_t} |\widetilde{\mathbf{b}}_x|^2 dxds,\quad\forall \epsilon>0.
  \end{split}
  \end{equation}
Since
   \begin{equation*}
\begin{split}
 & |\widetilde{u}(x,t)|= |u(x,t)-\overline u(x,t)|=\Big|\int_0^x\widetilde{u}_x dx\Big|\leq \left(\int_0^1\widetilde{u}_x^2dx\right)^{1/2}\omega^{1/2}(x),
 \forall x \in [0, 1/2],\\
 &|\widetilde{u}(x,t)|= |u(x,t)-\overline u(x,t)|=\Big|\int_x^1\widetilde{u}_x dx\Big|\leq \Big(\int_0^1\widetilde{u}_x^2dx\Big)^{1/2}\omega^{1/2}(x),
 \forall x \in [1/2, 1],\\
 \end{split}
  \end{equation*}
we have $
  |\widetilde{u}(x,t)|^2  \leq \left(\int_0^1\widetilde{u}_x^2dx\right) \omega(x).
  $
  Thus, from Lemma \ref{2.13} and \eqref{u14}, we obtain
   \begin{equation*}
\begin{split}
 \iint_{Q_t} \widetilde{u}^2|\overline{\mathbf{w}}_x|^2dxds
  \leq & \int_0^t\Big(\int_0^1\widetilde{u}_x^2dx\Big) \Big(\int_\Omega|\overline{\mathbf{w}}_x|^2\omega dx\Big)ds
  \leq    C \int_0^tD(s)\mathbb{H}(s)ds.
  \end{split}
  \end{equation*}
Substituting this into \eqref{0v33} completes the proof of \eqref{0v3}.\\

\indent{\bf Step 5} Claim that
 \begin{equation}\label{B}
\begin{split}
  & \int_\Omega |\widetilde{\mathbf{b}}|^2dx +
   \iint_{Q_t} |\widetilde{\mathbf{b}}_x|^2dxds\leq C\int_0^t D(s)\mathbb{H}(s)ds.
  \end{split}
  \end{equation}
    From the equations of $\mathbf{b}$ and $\overline{\mathbf{b}}$, we have that $   \widetilde{\mathbf{b}}_t+  \big(u \widetilde{\mathbf{b}}\big)_x
+  \big(\widetilde{u}\overline{\mathbf{b}}\big)_x- \widetilde{\mathbf{w}}_x-\nu\widetilde{\mathbf{b}}_{xx}=0.
  $
  Taking the inner product of the identity with $\widetilde{\mathbf{b}}$ and integrating over $Q_t$ give
  \begin{equation*}
\begin{split}
 \frac12\int_\Omega |\widetilde{\mathbf{b}}|^2dx +
  \nu\iint_{Q_t} |\widetilde{\mathbf{b}}_x|^2dxds
 & = -\frac12\iint_{Q_t}u_x|\widetilde{\mathbf{b}}|^2dxds+\iint_{Q_t} \big(\widetilde{u}  \overline{\mathbf{b}}\cdot\widetilde{\mathbf{b}}-\widetilde{\mathbf{b}}_x \cdot\widetilde{\mathbf{w}}\big)dxds\\
  &\leq  \frac{\nu}{2}\iint_{Q_t} |\widetilde{\mathbf{b}}_x|^2 dxds +C\int_0^t D(s)\mathbb{H}(s)ds.
  \end{split}
  \end{equation*}
  Thus, the  claim \eqref{B} follows.

  Combining the above five steps and taking a small
  constant $\epsilon>0$, we complete  the proof of Theorem 1.1(iii)
  by the Gronwall inequality. Therefore, the proof of Theorem 1.1 is  completed.

\section{Proof  of Theorem  1.3}

 From  $\sup\limits_{0<t<T}\int_\Omega\big(|\mathbf{w}_x|^2+|\overline{\mathbf{w}}_x|^2\big)\omega dx
 \leq C$  obtained in Theorem 1.1, we have
\begin{equation*}\label{equ7}
\begin{split}
  \sup\limits_{0<t<T}\Big(\int_\delta^{1-\delta}(|\mathbf{w}_x|^2 +|\overline{\mathbf{w}}_x|^2)dx\Big)
 \leq \frac{C}{\delta},\quad\forall \delta \in (0, 1/2).\\
\end{split}
\end{equation*}
Using the embedding theorem, Theorem 1.1(iii) and the H\"{o}lder  inequality, we have
\begin{equation*}
\begin{split}
 \|\mathbf{w}-\overline{\mathbf{w}}\|_{L^\infty(\delta,1-\delta)}^2\leq &C\Big(\int_\Omega|\mathbf{w}-\overline{\mathbf{w}}|^2dx
 +\int_\delta^{1-\delta}|\mathbf{w}-\overline{\mathbf{w}}||\mathbf{w}_x-\overline{\mathbf{w}}_x|dx\Big)\\
  \leq&C\sqrt{\mu}+C\Big(\int_\delta^{1-\delta}|\mathbf{w}-\overline{\mathbf{w}}|^2dx\int_\delta^{1-\delta}|\mathbf{w}_x-\overline{\mathbf{w}}_x|^2dx\Big)^{1/2}\\
  \leq &C\sqrt{\mu}+C\Big(\frac{\sqrt{\mu}}{\delta}\Big)^{1/2},\quad\forall \delta \in (0, 1/2).\\
\end{split}
\end{equation*}
This completes the proof of Theorem 1.3.

\section*{Acknowledgments}
 The research of this paper was  supported  by  the NSFC (11731008, 11571062, 11271381, 11431015),  the Program for Liaoning
Innovative Talents in University, Guangdong Natural Science
Fund (2014A030313161)   and  the Fundamental Research Fund  for the Central Universities (DMU).

\par

\end{document}